\pdfoutput=1

\documentclass{article}
\usepackage{natbib}

\usepackage{microtype}
\usepackage{graphicx}
\usepackage{booktabs} %
\usepackage{pgfplots}
\pgfplotsset{compat=1.17, filter discard warning=false}
\usepgfplotslibrary{groupplots}
\usepackage[hypertexnames=false]{hyperref}
\usepackage{xspace}

\usepackage{algorithm}
\usepackage{algorithmic}

\usepackage[preprint]{tmlr}
\AtBeginDocument{\hypersetup{
  pdftitle={Variable Metric Three Operator Splitting},
  pdfauthor={Hansol Suh, Tobin Isaac}}%
  \renewcommand{\headrulewidth}{0pt}}
\usepackage{amsmath}
\usepackage{amssymb}
\usepackage{mathtools}
\usepackage{amsthm}

\usepackage[capitalize,noabbrev]{cleveref}

\theoremstyle{plain}
\newtheorem{theorem}{Theorem}[section]
\newtheorem{proposition}[theorem]{Proposition}
\newtheorem{lemma}[theorem]{Lemma}
\newtheorem{corollary}[theorem]{Corollary}
\theoremstyle{definition}

\theoremstyle{remark}
\newtheorem{remark}[theorem]{Remark}

\DeclareMathOperator{\prox}{prox}

\DeclareMathOperator{\minimize}{minimize}
\DeclareMathOperator{\dom}{dom}

\newcommand{\bs}{\boldsymbol}
\newcommand{\va}{\boldsymbol{a}}
\newcommand{\bb}{\boldsymbol{b}}
\newcommand{\mm}{\boldsymbol{m}}
\newcommand{\xx}{\boldsymbol{x}}
\newcommand{\uu}{\boldsymbol{u}}
\newcommand{\vv}{\boldsymbol{v}}
\newcommand{\vs}{\boldsymbol{s}}
\newcommand{\zz}{\boldsymbol{z}}
\newcommand{\yy}{\boldsymbol{y}}
\newcommand{\MM}{\boldsymbol{M}}
\newcommand{\NN}{\boldsymbol{N}}
\newcommand{\DD}{\boldsymbol{D}}
\newcommand{\II}{\boldsymbol{I}}

\newcommand{\LL}{\mathcal{L}}
\newcommand{\SPD}{\mathbb{S}_{+\!+}^n}
\newcommand{\vmtos}{\cref{alg:vmtos}\xspace}
\newcommand{\Vmtos}{\Cref{alg:vmtos}\xspace}

\DeclareMathOperator*{\argmin}{arg\,min}

\newcommand{\algorithmicbreak}{\textbf{break}}
\newcommand{\BREAK}{\STATE \algorithmicbreak}

\usepackage[disable,textsize=tiny]{todonotes}

\title{Variable Metric Three Operator Splitting}

\date{}

\author{\name Hansol Suh \email hsuh@anl.gov \\
      \addr Mathematics and Computer Science Division \\
      Argonne National Laboratory, Lemont, IL, USA \\[2ex]
      \name Tobin Isaac \email toby.isaac@gmail.com \\
      \addr NVIDIA, Santa Clara, CA, USA}

\begin{document}

\maketitle

\begin{abstract}
    To minimize the sum $f + g + h$ of three convex functions when $f$ is differentiable with a Lipschitz-continuous gradient,
    we present a proximal algorithm, \emph{variable-metric three operator splitting} (VMTOS), a generalization of the three operator splitting (TOS) scheme of Davis and Yin.
    We establish convergence guarantees by controlling changes in the metric.
    The smallest saddle-point residual among the first $k$ primal-dual iterates is $\tilde{O}(k^{-1/2})$.
    When $f$ and the convex conjugate $h^*$ are strongly convex and the accepted metrics and their inverses are uniformly bounded, the rate improves to $\tilde{O}(\beta^k)$ for some $\beta \in (0,1)$.
    We develop a practical Barzilai--Borwein update for diagonal metrics that satisfies our bound on metric changes.
    Under the same strong convexity assumptions, this update guarantees the required uniform metric bounds, yielding R-linear convergence.
    Numerical experiments on benchmark optimization problems show that VMTOS is competitive with other proximal algorithms.
\end{abstract}

\section{Introduction}
\label{sec:introduction}

We consider the problem
\begin{equation}
\minimize_{\xx\in\mathbb{R}^n}\, f(\xx) + g(\xx) + h(\xx),
\label{eq:minimize}
\end{equation}
where $f:\mathbb{R}^n\to\mathbb{R}$ is differentiable and convex, and $g,h:\mathbb{R}^n\to(-\infty,+\infty]$ are proper, lower-semicontinuous convex functions.
This problem appears in the optimization of parametric models, where
$f(\xx)$ is the misfit or loss of a smooth model parameterized by $\xx$, and $g(\xx)$ and $h(\xx)$ represent some combination of nonsmooth penalties and constraints placed on $\xx$.

The three operator splitting (TOS) scheme of \citet{davis2017three} is a proximal algorithm that optimizes \eqref{eq:minimize} when $\nabla f$ is Lipschitz continuous,
but the Lipschitz constant $L$ must be known to choose the step size.
\citet{pedregosa2018adaptive} developed two adaptive variants of TOS that do not need $L$ to be known.
The variant with nonincreasing step sizes converges like the original scheme, but the variant that allows increasing step sizes is only shown to converge when $h$ is Lipschitz continuous with a known constant.

We present \emph{variable-metric three operator splitting} (VMTOS, \vmtos), a variant of TOS that uses, instead of a scalar step size, a variable metric
$\|\xx \|_{\MM_k} := \sqrt{\langle \MM_k \xx, \xx \rangle}$,
based on a different symmetric positive definite (s.p.d.) matrix $\MM_k$ in each iteration.

\begin{algorithm}
\caption{Variable Metric Three Operator Splitting Scheme}
\label{alg:vmtos}
\begin{algorithmic}[1]
\STATE{\bfseries Input:} 
     $\zz_0, \uu_0 \in \mathbb{R}^n$, 
       $\widehat\MM_0 \in \SPD$,
       $\rho \geq 0$,
       $\sigma \in (0,1)$,
       $\tau > 1$

\FOR{$k = 0, 1, 2, \dots$}
\STATE $\MM_k \gets \widehat\MM_k$ \hfill\COMMENT{initialize trial metric}
\LOOP \label{line:loop}
    \STATE $\xx_{k+1} \gets \argmin_{\xx}\, g(\xx) + \frac{1}{2} \| \zz_k - \MM_k^{-1}(\uu_k + \nabla f(\zz_k)) - \xx \|_{\MM_k}^2$\label{line:prox-g} \hfill\COMMENT{$g$ subproblem}
    \IF{$f(\xx_{k+1}) \leq  f(\zz_k) + \langle \nabla f(\zz_k), \xx_{k+1} - \zz_k \rangle + \frac{\sigma}{2}\|\xx_{k+1} - \zz_k\|_{\MM_k}^2$\label{line:sufficient-decrease}}
    \BREAK \hfill\COMMENT{sufficient decrease}
    \ELSE
    \STATE $\MM_k \gets \tau \MM_k$\label{line:strengthen-metric} \hfill\COMMENT{backtracking}
    \ENDIF
\ENDLOOP
\STATE     $\zz_{k+1} \gets \argmin_{\zz}\, h(\zz) + \frac{1}{2}\|\xx_{k+1} + \MM_k^{-1} \uu_k - \zz \|_{\MM_k}^2$\label{line:prox-h} \hfill\COMMENT{$h$ subproblem}

\STATE     $\uu_{k+1} \gets \uu_k + \MM_k(\xx_{k+1} - \zz_{k+1})$\label{line:dual-update}
\hfill \COMMENT{dual update}

\STATE     Set $d_{k+1} := 1 + \frac{\rho}{(k+1)(1 +\log(k+1))}$
\STATE     Choose $\widehat\MM_{k+1}$ such that $d_{k+1}^{-1} \MM_k \preceq \widehat\MM_{k+1} \preceq d_{k+1} \MM_k$\label{line:adapt-metric}
    \hfill \COMMENT{metric selection}
\ENDFOR
\end{algorithmic}
\end{algorithm}

\subsection{Our contribution}

In \cref{sec:convergence}, we analyze the convergence of the sequence
$\{(\xx_k, \uu_k)\}_{k=1}^\infty$ generated by \vmtos to a saddle point
of the Lagrangian $\LL(\xx, \uu)$ of the original minimization problem \eqref{eq:minimize}.
As a merit function for a pair $(\xx_k, \uu_k)$, we use the residuals of the subgradients $(\partial_{\xx} \LL(\xx_k, \uu_k), \partial_{\uu} (-\LL(\xx_k, \uu_k)))$, which contain zero only if $(\xx_k, \uu_k)$ is a saddle point of the Lagrangian.
Under conditions that we detail in \cref{thm:pointwise-convergence},
$\{(\xx_k, \uu_k)\}_{k=1}^\infty$ contains a subsequence whose residuals converge to zero at a rate of $\tilde{O}(k^{-1/2})$ (that is, within a polylogarithmic factor of $O(k^{-1/2})$).
In \cref{cor:quasilinear}, we show that when $f$ and $h^*$ are strongly convex and the accepted metrics are uniformly bounded and uniformly positive definite,
$\{(\xx_k, \uu_k)\}_{k=1}^\infty$ converges to the unique saddle point at a rate of $O((\log k)^\rho\beta^k)$ for some $\beta\in(0,1)$.
Under the same strong convexity assumptions, \cref{cor:practical-bb-linear} establishes R-linear convergence for \vmtos using the practical diagonal BB rule of \cref{sec:diagvm}.
Its proof uses \cref{lem:practical-bb-bounds}, which shows that this rule guarantees the required uniform metric bounds.

\Vmtos is not just a variable metric generalization of the adaptive algorithm of \citet{pedregosa2018adaptive}.  In their adaptive three operator splitting (ATOS) algorithm, a sufficient decrease condition is used that is equivalent to line \ref{line:sufficient-decrease} of \vmtos with $\sigma = 1$.  By choosing a fixed $\sigma<1$ and controlling consecutive metric proposals, we allow the metric to decrease (equivalent to allowing the step size to increase) without requiring $h$ to be Lipschitz.  Our general convex result controls best-iterate saddle-point residuals, whereas their sublinear analysis controls an ergodic Lagrangian gap.

Our analysis also draws on the work of \citet{GoncalvesEtAl2018} in proving the convergence of variable metric hybrid proximal extragradient iterations.  In their work it is assumed that the distances between consecutive metrics are \emph{summable}: that there exists $C_s < \infty$ and $\{c_k\}_{k=1}^\infty$ such that $\sum_k c_k \leq C_s$ and such that
\begin{equation}
(1 + c_{k+1})^{-1} \MM_k \preceq \MM_{k+1} \preceq (1 + c_{k+1}) \MM_k,
\label{eq:summable-differences}
\end{equation}
where $\NN \preceq \MM$ means $\MM - \NN$ is symmetric positive semidefinite.
This choice implies that $\{\MM_k\}$ stays within a compact neighborhood of $\MM_0$ and $\{\MM_k^{-1}\}$ within a compact neighborhood of $\MM_0^{-1}$.  For $\rho>0$, our proposal bound in line \ref{line:adapt-metric} of \vmtos uses $c_k = \rho / (k(1 + \log k))$, which is not a summable sequence,
but allows $\MM_k$ to become arbitrarily distant from $\MM_0$.  Backtracking introduces additional metric changes, which are accounted for in \cref{thm:metric-distance}.  Our general rate bounds include a polylogarithmic factor relative to similar bounds proved by Gon{\c{c}}alves et al.  We discuss the relationship between our bound on metric proposals and Barzilai--Borwein-type metric selection techniques in \cref{sec:metric-selection}.

\subsection{Related work}

In addition to the work of \citet{davis2017three} and \citet{pedregosa2018adaptive} on three operator splitting algorithms
and \citet{GoncalvesEtAl2018} on variable-metric ADMM, our work draws on ideas from other variable metric proximal algorithms.
In the variable-metric context, the proximal operator is defined to be
\begin{equation}\label{eq:prox-M-def}
\prox_{\MM,g}(\yy) := \argmin_{\xx}\ g(\xx) + \frac{1}{2}\|\xx - \yy\|_{\MM}^2,
\end{equation}
where $\MM$ is an s.p.d.\ matrix.
When used to modify the proximal gradient algorithm for $f(\xx)+g(\xx)$, the resulting approach is called variable-metric proximal gradient (VMPG).

Several VMPG variants have been proposed, using different update rules for $\MM_k$, for both convex  \citep{chouzenoux2014variable,salzo2017variable,lee2012proximal} and non-convex cases \citep{bonettini2016variable}.
Majorization-minimization criteria can be used to constrain the metric selection when the particular objective functions permit it \citep{chouzenoux2014variable}, but in the general case it is not always possible to guarantee that a metric defines a quadratic majorant.

\citet{park2020variable} note that fully variable metric methods (when $\MM$ may be any matrix in $\SPD$) fall short in practice because of the expense of dense matrix operations and the lack of directly computable proximal operators for general function/metric pairs \citep{lee2012proximal}.
They use diagonal variable metrics (as did \citealp{parente2008class}), which are compatible with a wide class of objective functions, 
and which can be inexpensively updated.
They give details for implementing a Barzilai--Borwein \citep{barzilai1988two} update condition entry-by-entry on the diagonal.
We will extend that approach to include our constraints on metric selection in \cref{sec:metric-selection}.

\section{Background, notation, and assumptions}\label{sec:background}

Unsubscripted vector norms are Euclidean, unsubscripted matrix norms are spectral, and $\|\cdot\|_F$ denotes the Frobenius norm.  We write $\mathbb{S}^n$ for the real symmetric matrices and $\SPD$ for the positive definite ones.

\paragraph{Notation for convex functions}

We use the following standard notation from the convex analysis literature \citep[see, e.g.,][]{Bauschke2017}.
\begin{itemize}\setlength\itemsep{0em}

    \item $\Gamma_0(\mathbb{R}^n)$: the class of closed, proper, convex, extended-valued functions on $\mathbb{R}^n$.
    In subsequent definitions it is assumed $\varphi \in \Gamma_0(\mathbb{R}^n)$.
    
    \item $\dom \varphi$: the effective domain of $\varphi$.
    
    \item $\varphi^*$: the convex conjugate of $\varphi$,
    which satisfies
\begin{equation}
\varphi^* \in \Gamma_0(\mathbb{R}^n), \quad \varphi^{**} = \varphi.\label{eq:double-conjugate}
\end{equation}

    \item  $\partial \varphi:\mathbb{R}^n \rightrightarrows \mathbb{R}^n$: the subdifferential of $\varphi$, a monotone operator,
\begin{equation}\label{eq:monotone}
    \langle \uu - \vv, \xx - \yy \rangle \geq 0 \quad\text{for all }\uu \in \partial \varphi(\xx), \vv \in \partial \varphi(\yy).
\end{equation}
    The subgradients of convex conjugates are inverses,
\begin{equation}\label{eq:subdiff-inverse}
    \uu \in \partial \varphi(\xx) \Leftrightarrow \xx \in \partial \varphi^*(\uu).
\end{equation}
By the optimality conditions of \eqref{eq:prox-M-def}, 
\begin{equation}\label{eq:prox-to-subdiff}
    \xx = \prox_{\MM,\varphi}(\yy) \Leftrightarrow \MM(\yy - \xx) \in \partial \varphi(\xx).
\end{equation}

\end{itemize}

\paragraph{Duality and the Lagrangian}

We let $P(\xx) := f(\xx) + g(\xx) + h(\xx)$ be the objective function of \eqref{eq:minimize}.
Let $\Omega:=\dom g\times\dom h^*$.  For the perturbation function $F(\xx,\zz) := f(\xx) + g(\xx) + h(\xx + \zz)$, define the finite-valued Lagrangian $\LL:\Omega\to\mathbb{R}$ by
\begin{equation}
\begin{split}
    \LL(\xx,\uu) :&= \inf_{\zz\in\mathbb{R}^n}\, F(\xx,\zz) - \langle \uu, \zz \rangle \\
    &= f(\xx) + g(\xx) + \langle \uu, \xx \rangle - h^*(\uu).\label{eq:lagrangian}
\end{split}
\end{equation}
For $(\xx,\uu)\in\Omega$, the sections $\LL(\cdot,\uu)$ and $-\LL(\xx,\cdot)$ are extended by $+\infty$ outside $\dom g$ and $\dom h^*$, respectively.
The dual maximization objective is
\begin{equation}
   D(\uu) := -(f+g)^*(-\uu) - h^*(\uu).
\end{equation}
For $\uu\in\dom h^*$, $D(\uu)=\inf_{\xx\in\dom g}\LL(\xx,\uu)$.
We denote by $S(\LL)$ the set of saddle points of $\LL$,
\begin{equation}\label{eq:saddle-point}
    S(\LL) := \bigl\{(\xx^\star, \uu^\star)\in\Omega :
    \LL(\xx^\star, \uu) \leq \LL(\xx^\star, \uu^\star) \leq \LL(\xx, \uu^\star)
    \text{ for all } (\xx, \uu) \in \Omega\bigr\}.
\end{equation}

The Lagrangian is convex-concave on $\Omega$, and a pair $(\xx^\star, \uu^\star)\in\Omega$ is a saddle point if and only if
\begin{align}
    0 &\in \partial_{\xx} \LL(\xx^\star, \uu^\star) = \nabla f(\xx^\star) + \partial g(\xx^\star) + \uu^\star,\label{eq:dLdx} \\
    0 &\in \partial_{\uu} (-\LL(\xx^\star, \uu^\star)) = \partial h^*(\uu^\star) - \xx^\star.\label{eq:dLdu}
\end{align}

The distance of the zero vector from these subdifferentials is a measure of the suboptimality of a potential saddle point $(\xx,\uu)\in\Omega$,
\begin{equation}\label{eq:residual-def}
    r(\xx, \uu) := \inf_{\substack{\vv \in \partial_{\xx} \LL(\xx,\uu), \\ \yy \in \partial_{\uu}(-\LL(\xx, \uu))}}\,\sqrt{\|\vv\|^2 + \|\yy\|^2}.
\end{equation}

\paragraph{Regularity assumptions}

We assume the functions in \eqref{eq:minimize} are closed, proper convex functions,
\begin{equation}
f, g, h \in \Gamma_0(\mathbb{R}^n),\tag{A1}\label{eq:proper}
\end{equation}
that $f:\mathbb{R}^n\to\mathbb{R}$ is differentiable with $L$-Lipschitz gradient (or $L$-smooth) for some $L\geq0$,
\begin{equation}
\|\nabla f (\xx) - \nabla f(\yy)\| \leq L \|\xx - \yy\| \quad \text{for all }\xx, \yy \in \mathbb{R}^n,
\tag{A2}\label{eq:L-lipschitz}
\end{equation}
and that at least one saddle point exists,
\begin{equation}
    S(\LL) \neq \emptyset.\tag{A3}\label{eq:solutions-exist}
\end{equation}
Under these assumptions, \eqref{eq:double-conjugate} and \eqref{eq:lagrangian} imply
$P(\xx) = \sup_{\uu\in\dom h^*} \LL(\xx,\uu)$ for $\xx\in\dom g$.  Therefore
$(\xx^\star, \uu^\star) \in S(\LL)$ implies by definition \eqref{eq:saddle-point}
\begin{equation}
    P(\xx^\star) = \inf_{\xx\in\dom g} \LL(\xx, \uu^\star) = \sup_{\uu\in\dom h^*} \LL(\xx^\star, \uu) = D(\uu^\star),
\end{equation}
which means that strong duality holds and that $\xx^\star$ and $\uu^\star$ are optimizers of the primal and dual problems.
There are several regularity conditions that can be verified \textit{a~priori} that imply \eqref{eq:solutions-exist} \citep{Bot2011}: for instance, if the relative interiors of $\dom(f+g) - \dom h$ and $\dom(f+g)^* + \dom h^*$ both contain the zero vector.

\section{Convergence}\label{sec:convergence}

\paragraph{Behavior within one iteration}

First, we can be certain that the backtracking loop on 
lines \ref{line:loop}--\ref{line:strengthen-metric} exits after a finite number of iterations.

\begin{proposition}\label{prop:backtrack}
    Let \eqref{eq:proper} and \eqref{eq:L-lipschitz} hold.  Then in each iteration of \vmtos the sufficient decrease condition on line \ref{line:sufficient-decrease} is satisfied after
    a finite number of iterations of the backtracking loop on line \ref{line:loop}.
\end{proposition}
\begin{proof} See \cref{proof:backtrack}. \end{proof}
We can rephrase the assignments after the conclusion of the backtracking loop in proximal notation,
giving us the following equations for every $k$:
\begin{align}
    \xx_{k+1} &= \prox_{\MM_k,g}(\zz_k - \MM_k^{-1}(\uu_k +  \nabla f(\zz_k))), \label{eq:tos-iter-prox-g}\\
    \zz_{k+1} &= \prox_{\MM_k,h}(\xx_{k+1} + \MM_k^{-1}\uu_k), \label{eq:tos-iter-prox-h} \\
    \uu_{k+1} &= \uu_k + \MM_k (\xx_{k+1} - \zz_{k+1}). \label{eq:tos-iter-dual-update}
\end{align}

First we note that these equations show that a bound on $r(\xx_{k+1},\uu_{k+1})$ can be computed in each iteration.
\begin{proposition}
Let (\ref{eq:proper},\ref{eq:L-lipschitz}) hold.  Then after iteration $k$ of \vmtos
\begin{equation}\label{eq:first-r-bound}
\begin{split}
&r(\xx_{k+1},\uu_{k+1})^2 \leq \\
&\| \MM_k(\zz_k - \zz_{k+1}) + \nabla f(\xx_{k+1}) - \nabla f(\zz_k)\|^2 + \\
&\|\MM_k^{-1}(\uu_k - \uu_{k+1})\|^2.
\end{split}
\end{equation}
\end{proposition}

\begin{proof} Application of \eqref{eq:prox-to-subdiff} to \eqref{eq:tos-iter-prox-g} and \eqref{eq:tos-iter-prox-h}, algebraic manipulation, and the subgradient relationship \eqref{eq:subdiff-inverse}. See \cref{lem:residual-bound} for details. \end{proof}

Next we note that equations \eqref{eq:tos-iter-prox-g}--\eqref{eq:tos-iter-dual-update} and the sufficient decrease condition guarantee progress towards a saddle point in $S(\LL)$ in the $\MM_k$ norm in each iteration.

\begin{lemma}\label{lem:keyrecursive}
Let assumptions (\ref{eq:proper},\ref{eq:L-lipschitz},\ref{eq:solutions-exist}) hold.
Let $f$ be $\mu_f$-strongly convex and $h^*$ be $\mu_h$-strongly convex (where $\mu_f$ and $\mu_h$ may be zero). 
Then after iteration $k$ of \vmtos for every $(\xx,\uu)\in\Omega$ we have
\begin{equation}
\begin{split}
&2(\mathcal{L}(\xx_{k+1},\uu)-\mathcal{L}(\xx,\uu_{k+1}))
+(1 - \sigma) \|\xx_{k+1} - \zz_k\|_{\MM_k}^2\\
&+ \mu_f \|\zz_k-\xx\|_2^2 + \mu_h \|\uu_{k+1} - \uu\|_2^2 \\
& \leq  \|\zz_k-\xx\|_{\MM_k}^2  +\|\uu_k-\uu\|_{\MM_{k}^{-1}}^2 \\
& - \|\zz_{k+1}-\xx\|_{\MM_{k}}^2 - \|\uu_{k+1} - \uu\|_{\MM_k^{-1}}^2.\label{eq:key-recursive}
\end{split}
\end{equation}
\end{lemma}

\begin{proof}See \cref{proof:keyrecursive}.\end{proof}

\begin{corollary}\label{cor:keyrecursive-solution}
Under the same assumptions as \cref{lem:keyrecursive}, for every $(\xx^\star, \uu^\star)\in S(\LL)$ and every $k$,
\begin{equation}
\begin{split}
&    \|\zz_{k+1} - \xx^\star\|_{\MM_k}^2 +
    \|\uu_{k+1} - \uu^\star\|_{\MM_k^{-1}}^2
    \leq \\
&    \|\zz_k - \xx^\star\|_{\MM_k}^2 +
    \|\uu_k - \uu^\star\|_{\MM_k^{-1}}^2    -\\
&    (1 - \sigma) \|\xx_{k+1} - \zz_k\|_{\MM_k}^2.
\end{split}
\end{equation}
\end{corollary}
 
\begin{proof} For every $(\xx^\star, \uu^\star) \in S(\LL)$, $\LL(\xx,\uu^\star) - \LL(\xx^\star,\uu) \geq 0$ for all $(\xx,\uu)\in\Omega$.  Applying \eqref{eq:prox-to-subdiff} to the proximal steps, together with the dual update, gives $(\xx_{k+1},\uu_{k+1})\in\Omega$.\end{proof}

Inequality \eqref{eq:key-recursive} is almost a straightforward generalization to variable metrics of the single-iteration inequality used in the proof of \citet[Lemma~4]{pedregosa2018adaptive},
\begin{equation}
\begin{split}
&2\gamma_k(\mathcal{L}(\xx_{k+1},\uu)-\mathcal{L}(\xx,\uu_{k+1})) + \gamma_k \mu_f \|\zz_k-\xx\|_2^2\\
&+ \gamma_k \mu_h \|\uu_{k+1} - \uu\|_2^2  \leq \|\zz_k-\xx\|^2 + \gamma_k^2\|\uu_k-\uu\|^2  \\
& - \|\zz_{k+1}-\xx\|^2- \gamma_k^2\|\uu_{k+1} - \uu\|^2.\label{eq:key-recursive-atos}
\end{split}
\end{equation}
Setting aside the differences in metric selection and backtracking, an iteration of ATOS with step size $\gamma_k$ is equivalent to an iteration of \vmtos with $\MM_k = \gamma_k^{-1} \II$.

There are two important distinctions between the inequalities \eqref{eq:key-recursive} and \eqref{eq:key-recursive-atos}.
First, the term $(1-\sigma)\|\xx_{k+1} - \zz_k\|_{\MM_k}^2$ on the left-hand side of \eqref{eq:key-recursive} is not present in \eqref{eq:key-recursive-atos} because of the less restrictive definition of sufficient decrease in ATOS, equivalent to the choice $\sigma = 1$.

Second, the right-hand side terms of \eqref{eq:key-recursive} are measured in the $\MM_k$ and $\MM_k^{-1}$ norms, rather than the $\ell_2$-norm as in \eqref{eq:key-recursive-atos}.  When the metrics vary with $k$, summing the inequality over $k$ does not telescope the right-hand side, and there is not a weighted sum that will telescope the right-hand side for all $\xx$ and $\uu$.

\paragraph{Bounding metric equivalences}

To relate the inequality \eqref{eq:key-recursive} across iterations, we must relate the $\MM_k$-{} and $\MM_{k}^{-1}$-norms to the $\MM_{k-1}$-{} and $\MM_{k-1}^{-1}$-norms.  Let us define
\begin{equation}\label{eq:big-delta-def}
    \Delta(\MM,\NN) := \min\{d \geq 1: d^{-1}\MM \preceq \NN \preceq d\MM\}.
\end{equation}
This quantity is submultiplicative, so for any $j\leq k$,
\begin{equation}\label{eq:delta-def}
    \Delta(\MM_j,\MM_k) \leq \prod_{i=j}^{k-1} \Delta(\MM_i,\MM_{i+1}) =: \Delta_{j,k}.
\end{equation}
If backtracking never occurs, then the metric selection on line \ref{line:adapt-metric} enforces $\Delta_{k-1,k} \leq d_k = 1 + \rho/(k(1 + \log k))$.  The product of these terms is $O((\log k)^\rho)$ (see \cref{lem:d-prod}), and when backtracking does occur $\Delta_{0,k}$ is still polylogarithmic in $k$.

\begin{theorem}\label{thm:metric-distance}
Let assumptions (\ref{eq:proper},\ref{eq:L-lipschitz}) hold, and for each $k\geq 0$ let $\MM_k\in \SPD$ be the final value of the matrix after backtracking in \vmtos. Then $\Delta_{0,k} \in O((\log k)^{2\rho})$.
\end{theorem} 

\begin{proof} See \cref{proof:metric-distance}.\end{proof}

\paragraph{Pointwise sublinear convergence}

\Cref{thm:metric-distance} gives us a bound we can use to combine the inequality \eqref{eq:key-recursive}
from \cref{lem:keyrecursive} across iterations.  To simplify the statements of our convergence results,
we make the following assumption on the initial point:
\begin{equation}
    \uu_0 \in \partial h(\zz_0).\tag{A4}\label{eq:u0-in-partial-h-z0}
\end{equation}
This assumption can be enforced by selecting an arbitrary $\yy_0 \in \mathbb{R}^n$ and computing
\[
\begin{aligned}
    \zz_0 &\gets \prox_{\widehat\MM_0,h} (\yy_0), \\
    \uu_0 &\gets \widehat\MM_0(\yy_0 - \zz_0).
\end{aligned}
\]

Otherwise, the statements of our convergence results can be modified to relate subsequent iterations to $k=1$ instead of $k=0$.  For the statement of this and later results, we need to define the weighted distance of $(\xx,\uu)$ from the solution set,
\begin{equation}\label{eq:strong-dist}
d((\xx,\uu))^2 :=
\min_{(\xx^\star,\uu^\star)\in S(\LL)}
\|\xx - \xx^\star\|_{\MM_0}^2 + \|\uu - \uu^\star\|_{\MM_0^{-1}}^2.
\end{equation}

\begin{theorem}\label{thm:pointwise-convergence}
Let assumptions (\ref{eq:proper},\ref{eq:L-lipschitz},\ref{eq:solutions-exist},\ref{eq:u0-in-partial-h-z0}) hold.
After iteration $k$ of \vmtos there exists some $j$ with $1 \leq j\leq k+1$ such that
\[
r(\xx_j,\uu_j) \leq \Delta_{0,k} d((\zz_0,\uu_0)) \sqrt{\frac{C_p}{k+1}} ,
\]
where $C_p = (\max\{2\|\MM_0\|,\|\MM_0^{-1}\|\} + 2L^2\|\MM_0^{-1}\|) / (1 - \sigma)$.
\end{theorem}

\begin{proof} See \cref{proof:pointwise-convergence}. \end{proof}

The proof of \cref{thm:pointwise-convergence} places no assumptions on the size of $\Delta_{0,k}$, but if the metric selection bound in \vmtos is followed, $\tilde{O}(k^{-1/2})$ convergence occurs.

\begin{corollary}
Let the assumptions of \cref{thm:pointwise-convergence} hold and suppose the condition on line \ref{line:adapt-metric} is followed. Let $r_k:=\min_{1\leq j\leq k+1}r(\xx_j,\uu_j)$ be the smallest residual after iteration $k$ of \vmtos.  Then $r_k\in O((\log k)^{2\rho}/\sqrt{k})$ as $k\to\infty$.
\end{corollary}

\paragraph{Geometric convergence with strong convexity}

If the monotone operator defined by the saddle-point optimality conditions (\ref{eq:dLdx},\ref{eq:dLdu}) is strongly monotone, then many fixed point iterations converge linearly to the unique saddle point.
Under strong convexity of $f$ and $h^*$, we first establish R-linear convergence assuming uniform bounds on the accepted metrics and their inverses.
We then prove in \cref{sec:diagvm} that the practical diagonal BB rule guarantees these bounds.

When $f$ and $h^*$ are both strongly convex, it is reasonable to assume that the metrics will be selected within some bounded range: for constants $0<c\leq C<\infty$,
\begin{equation}
    c \II \preceq \MM_k \preceq C \II \quad \text{for all } k \geq 0.\tag{A5}\label{eq:metric-double-bounds}
\end{equation}

\begin{theorem}\label{thm:linear-convergence}
Let assumptions (\ref{eq:proper},\ref{eq:L-lipschitz},\ref{eq:solutions-exist},\ref{eq:metric-double-bounds}) hold.  Assume additionally that $f$ is $\mu_f$-strongly convex and $h^*$ is $\mu_h$-strongly convex for $\mu_f, \mu_h > 0$.
Let $(\xx^\star,\uu^\star)$ be the unique saddle point.  We define
\begin{align}
E_k &:= \|\zz_k - \xx^\star\|_{\MM_k}^2 + (1-\theta)^{-1}\|\uu_k - \uu^\star\|_{\MM_{k}^{-1}}^2, \\
\theta &:= c\mu_h / (1 + c\mu_h), \\
\phi   &:= C^{-1} \mu_f.
\end{align}
Let $(\zz_{k+1}, \uu_{k+1})$ be the iterate produced by \vmtos after $k+1$ iterations. Then,
\begin{equation}\label{eq:quasilinear-bound}
    E_{k+1} \leq \Delta_{0,k+1} (1-\min\{\phi, \theta\})^{k+1} E_0.
\end{equation}
\end{theorem}

\begin{proof}
See \cref{proof:linear-convergence}.
\end{proof}

\begin{corollary}\label{cor:quasilinear}
Let the assumptions of \cref{thm:linear-convergence} hold, let the proposed metrics be selected according to line \ref{line:adapt-metric} of \vmtos, and let $\beta:=\sqrt{1-\min\{\phi,\theta\}}\in(0,1)$.  Then
\[
d((\zz_k,\uu_k))\in O((\log k)^\rho\beta^k).
\]
The same rate holds for $\|\xx_k-\xx^\star\|+\|\uu_k-\uu^\star\|$ and for $r(\xx_k,\uu_k)$.
\end{corollary}

\begin{proof}
Assumption \eqref{eq:metric-double-bounds} implies $\Delta(\MM_0,\MM_k)\leq C/c$.  Thus \eqref{eq:big-delta-def}, \eqref{eq:quasilinear-bound}, and \cref{thm:metric-distance} give
\begin{align*}
d((\zz_k,\uu_k)) &\leq
\sqrt{\Delta(\MM_0,\MM_{k}) E_k}
\\
&\leq
\sqrt{(C/c)\Delta_{0,k}E_0}\,\beta^k
\\
&\in O((\log k)^\rho\beta^k).
\end{align*}%
Since the saddle point is unique, this also bounds the Euclidean errors in $\zz_k$ and $\uu_k$.  By \eqref{eq:tos-iter-dual-update} and $\|\MM_k^{-1}\|\leq c^{-1}$,
\[
\|\xx_{k+1}-\xx^\star\|
\leq\|\zz_{k+1}-\xx^\star\|
+c^{-1}\bigl(\|\uu_{k+1}-\uu^\star\|+\|\uu_k-\uu^\star\|\bigr).
\]
This gives the same rate for $\xx_k$.  Finally, \eqref{eq:first-r-bound}, $L$-smoothness, and $\|\MM_k\|\leq C$ give the residual rate.
\end{proof}

Since $(\log k)^\rho\beta^k=O(q^k)$ for every $q\in(\beta,1)$, the iterates and residuals converge R-linearly.

\section{Barzilai--Borwein metric selection}\label{sec:metric-selection}

\Vmtos specifies bounds on the proposed metrics used in the convergence proofs in the previous section, but otherwise allows freedom in their selection.

A popular way to compute the metric in proximal gradient methods, which we adopt here, is via the Barzilai--Borwein (BB) method, which approximates the secant condition,
\begin{equation} \label{secant}
\widehat\MM_{k+1} \vs_k \approx \yy_k,
\end{equation}
where $\vs_k:=\zz_{k+1}-\zz_k$ and $\yy_k:=\nabla f(\zz_{k+1})-\nabla f(\zz_k)$
\citep{fletcher2005barzilai}.

If the space of metrics includes all s.p.d.\ matrices and $\langle\vs_k,\yy_k\rangle>0$, then
\eqref{secant} can be satisfied exactly by, for example, a BFGS update.
When we restrict the choice of metric, for example to diagonal matrices, a regularized least-squares problem can be solved instead.
We show that when such an approach is used, the distance between metrics is small in a way that is bounded by the norm of the update $\vs_k$.

\begin{lemma}\label{lem:bb-closeness}
Let $\mathcal{C}$ be a closed convex subset of $\SPD$ containing $\MM_k$, let $\vs_k := \zz_{k+1} - \zz_k$, and let $\yy_k := \nabla f(\zz_{k+1}) - \nabla f(\zz_k)$.
Let $\omega > 0$ be given, and suppose $\widehat\MM_{k+1}$ satisfies
\begin{equation}\label{eq:bb-obj}
\widehat\MM_{k+1} := \argmin_{\MM \in \mathcal{C}} \tfrac{1}{2}\|\MM \vs_k - \yy_k\|_2^2 + \tfrac{\omega}{2} \|\MM - \MM_k\|_F^2.
\end{equation}
If $f$ is $L$-smooth, then for $\|\vs_k\|_2$ sufficiently small we have the bound
\[
\Delta(\MM_k,\widehat\MM_{k+1}) \leq \left(1 - \frac{L + \|\MM_k\|_2}{\omega\lambda_{\min}(\MM_k)}\|\vs_k\|^2\right)^{-1}.
\]
\end{lemma}

\begin{proof}
See \cref{proof:bb-closeness}.
\end{proof}

With \cref{lem:bb-closeness}, we can construct a constrained optimization problem that bounds $\Delta(\MM_k,\widehat\MM_{k+1})$.  For this problem, let $\mathbb{K}\subseteq\mathbb{S}^n$ be closed and convex with $\MM_k\in\mathbb{K}$ for every $k\geq0$:

\begin{equation}\label{eq:bb-obj-with-bounds}
\begin{aligned}
\widehat\MM_{k+1} &:= \argmin_{\MM \in \mathbb{K}\cap\SPD} \tfrac{1}{2}\|\MM \vs_k - \yy_k\|_2^2 + \tfrac{\omega}{2} \| \MM - \MM_{k}\|_F^2, \\
&\phantom{:=}{} \text{ such that }\Delta(\MM_k,\MM) \leq d_{k+1},\quad
\text{(as in line \ref{line:adapt-metric})}.
\end{aligned}
\end{equation}

When the assumptions of \cref{thm:linear-convergence} hold and $\omega>0$ is fixed, we can show that this constrained version of the Barzilai--Borwein criterion removes the polylogarithmic factor from \cref{cor:quasilinear}, because the metric changes $\Delta_{k,k+1}-1$ are summable.

\begin{corollary}\label{cor:truly-linear}
    Let the assumptions of \cref{thm:linear-convergence} hold, and suppose that the proposed metric $\widehat\MM_{k+1}$ is selected according to \eqref{eq:bb-obj-with-bounds} for every $k\geq0$, with a fixed $\omega>0$.  Let $\beta:=\sqrt{1-\min\{\phi,\theta\}}$ as in \cref{cor:quasilinear}.  Then
    \[
        d((\zz_k,\uu_k)) \in O(\beta^k).
    \]
\end{corollary}

\begin{proof}
    See \cref{proof:truly-linear}.
\end{proof}

\subsection{Diagonal Variable Metric Selection}\label{sec:diagvm}
A practical choice for $\mathbb{K}$ in \eqref{eq:bb-obj-with-bounds} is the space of diagonal matrices.
Let $\MM_k=\mathrm{diag}(\mm_k)$ be the accepted metric, let $m_k^i$ and $\widehat m_k^i$ be the $i$th diagonal entries of $\MM_k$ and $\widehat\MM_k$, respectively, and let $s_k^i$ and $y_k^i$ be the $i$th entries of $\vs_k$ and $\yy_k$ defined above.  For diagonal metrics, the condition $\Delta(\MM_k,\widehat\MM_{k+1}) \leq d_{k+1}$ is equivalent to
\[
\frac{m_k^i}{d_{k+1}}\leq \widehat m_{k+1}^i\leq d_{k+1}m_k^i
\qquad(1\leq i\leq n).
\]

With this choice of $\mathbb{K}$, \eqref{eq:bb-obj-with-bounds} can be solved elementwise.
For $\omega>0$, and without constraints, the unique minimizer is
\[
p_k^i:=\argmin_{t\in\mathbb{R}}
\left\{\tfrac12(ts_k^i-y_k^i)^2+\tfrac{\omega}{2}(t-m_k^i)^2\right\}
=\frac{s_k^i y_k^i+\omega m_k^i}{(s_k^i)^2+\omega}.
\]
To enforce the bounds above, define $\operatorname{clip}_{[a,b]}(t):=\min\{b,\max\{a,t\}\}$ for $a\leq b$.  The solution of \eqref{eq:bb-obj-with-bounds} is then
\[
\widehat m_{k+1}^i
=\operatorname{clip}_{[m_k^i/d_{k+1},\,d_{k+1}m_k^i]}(p_k^i).
\]
The feasible set contains $\MM_k$, as required by \cref{lem:bb-closeness}.  Applying \cref{cor:truly-linear} with fixed $\omega>0$ also requires the uniform bounds on accepted metrics in \eqref{eq:metric-double-bounds}.
We next describe the practical BB clipping rule we use and prove that,
under the strong convexity assumption on $f$, its accepted metrics satisfy the uniform bounds in \eqref{eq:metric-double-bounds}.

\paragraph{Practical BB clipping rule}
We use an iteration-dependent regularization weight $\omega_k>0$, replacing $\omega$ by $\omega_k$ in the formula for $p_k^i$.
The metric update additionally uses scalar BB curvature estimates.  For $\langle\vs_k,\yy_k\rangle>0$, the BB step-size estimates are
\begin{equation}\label{bb}
\begin{split}
    \alpha_{\text{BB}1,k} &:= \|\vs_k\|_2^2 / \langle \vs_k,\yy_k \rangle,\\
    \alpha_{\text{BB}2,k} &:= \langle \vs_k, \yy_k \rangle / \|\yy_k\|_2^2. %
    \end{split}
\end{equation}

Set $a_k=(\alpha_{\text{BB}1,k})^{-1}$ and $b_k=(\alpha_{\text{BB}2,k})^{-1}$, so $0<a_k\leq b_k$ by Cauchy--Schwarz.  These estimate curvature along $\vs_k$.
While the BB interval $[a_k,b_k]$ provides a heuristic curvature range, each proposed diagonal entry must lie in the
closeness interval $[m_k^i/d_{k+1},d_{k+1}m_k^i]$, as required by line \ref{line:adapt-metric}.
For each coordinate $i$ with compatible BB and closeness bounds, we choose the proposed diagonal entry as the unique solution of the following regularized least-squares problem
\begin{equation}\label{bbopt}
\begin{aligned}
\widehat m_{k+1}^i
&:=\argmin_{t\in\mathbb{R}}
\left\{\frac12(ts_k^i-y_k^i)^2
+\frac{\omega_k}{2}(t-m_k^i)^2\right\}\\
&\text{subject to}\quad
a_k\leq t\leq b_k,\quad
\frac{m_k^i}{d_{k+1}}\leq t\leq d_{k+1}m_k^i.
\end{aligned}
\end{equation}

For a coordinate with disjoint intervals, \eqref{bbopt} is infeasible (see \Cref{rem:bb-clipping-counterexample}).
In this case, we compute each proposed diagonal entry by first clipping $p_k^i$ to the BB interval, then clipping the result to the
closeness interval:
\begin{equation}\label{eq:sequential-clipping}
\widehat m_{k+1}^i
=\operatorname{clip}_{[m_k^i/d_{k+1},\,d_{k+1}m_k^i]}
\bigl(\operatorname{clip}_{[a_k,b_k]}(p_k^i)\bigr).
\end{equation}
When the two intervals intersect, \eqref{eq:sequential-clipping} returns the solution of
\eqref{bbopt}: the objective in \eqref{bbopt} is a strictly convex quadratic in $t$ with
unconstrained minimizer $p_k^i$, so its constrained minimizer is $p_k^i$ clipped to the
intersection, which is what the two successive clips produce.  Hence
\eqref{eq:sequential-clipping} defines the proposal for every coordinate.

\paragraph{Convergence of the practical rule}
Under strong convexity of $f$, the BB interval and sequential clipping provide the uniform bounds on accepted metrics required by \eqref{eq:metric-double-bounds}, including the effects of backtracking.

\begin{lemma}[Uniform bounds for practical BB clipping]\label{lem:practical-bb-bounds}
Let \eqref{eq:proper} and \eqref{eq:L-lipschitz} hold, and assume that $f$ is $\mu_f$-strongly convex for $\mu_f>0$.
Suppose \vmtos starts with a diagonal accepted metric $\MM_0\in\SPD$ and selects proposals by \eqref{eq:sequential-clipping}, with any sequence $\omega_k>0$.
When $\vs_k=0$, reuse the most recent valid BB interval, with a fixed initial interval $[a_{\mathrm{init}},b_{\mathrm{init}}]\subset(0,\infty)$ before the first nonzero secant.
Then the accepted metrics satisfy \eqref{eq:metric-double-bounds} for constants $0<c\leq C<\infty$ independent of $k$.
\end{lemma}

\begin{proof}
See \cref{proof:practical-bb-bounds}.
\end{proof}

\begin{corollary}[R-linear convergence of the practical rule]\label{cor:practical-bb-linear}
Let the assumptions of \cref{lem:practical-bb-bounds} and \eqref{eq:solutions-exist} hold, and assume that $h^*$ is $\mu_h$-strongly convex for $\mu_h>0$.
Let $c,C$ be the uniform metric bounds from that lemma and set
\[
\beta:=\sqrt{1-\min\left\{\frac{\mu_f}{C},\frac{c\mu_h}{1+c\mu_h}\right\}}\in(0,1).
\]
For every $q\in(\beta,1)$, the iterates and residuals satisfy
\[
\|\xx_k-\xx^\star\|+\|\uu_k-\uu^\star\|\in O(q^k),
\qquad r(\xx_k,\uu_k)\in O(q^k).
\]
\end{corollary}

\begin{proof}
\Cref{lem:practical-bb-bounds} supplies \eqref{eq:metric-double-bounds}, and the outer clipping in \eqref{eq:sequential-clipping} enforces line \ref{line:adapt-metric} of \vmtos.
Thus \cref{cor:quasilinear} bounds both errors by $O((\log k)^\rho\beta^k)$.
For each $q\in(\beta,1)$, $(\log k)^\rho(\beta/q)^k\to0$, giving the claimed rates.
\end{proof}

This R-linear guarantee holds for varying positive $\omega_k$, and \cref{lem:practical-bb-bounds} guarantees the required uniform metric bounds.
The sharper $O(\beta^k)$ bound in \cref{cor:truly-linear} concerns \eqref{eq:bb-obj-with-bounds}, whose feasible set contains $\MM_k$.

\section{Numerical Experiments}\label{sec:experiments}
In this section we compare VMTOS using diagonal Barzilai--Borwein metric selection with TOS and ATOS.  All methods are implemented in a forked branch of the \emph{copt} package \citep{copt}.
\makeatletter
\if@accepted
The code for the experiments is available upon request.
\else
The source code and data for reproducing the experiments are included in the supplementary materials.
\fi
\makeatother

We consider tomography imaging with total variation regularization, overlapping group lasso, portfolio optimization, and a synthetic anisotropic example.
We take $P^\star$ to be a dual lower bound on the optimal value; for the portfolio problems, we instead measure the distance to the optimal solution $\xx^\star$.

\paragraph{Solver configuration.}
For all experiments, all solvers start from the step size $1/L$.  VMTOS uses BB regularization $\omega_k = \|\vs_k\|^2/n+\varepsilon_{\min}$, $\sigma=0.99$, and $\tau=1/0.7$, where $\varepsilon_{\min}>0$ is a small positive constant to avoid division by zero in the metric update.
We use $\rho=10$ throughout.
We use BB clipping as described in \cref{sec:diagvm}, without imposing an additional upper bound on the accepted metrics.
We set the maximum step-size growth factor for ATOS to $1.02$.
For residual monitoring and stopping, we use the square root of the right-hand side of \eqref{eq:first-r-bound}, with the tolerance applied to its current value.

\subsection{Tomography Image Recovery with Total Variation Regularization}\label{sec:e1}
We recover a $50 \times 50$ Shepp--Logan phantom \citep{shepp1974fourier} from simulated X-ray computed tomography (CT) measurements by solving
\begin{equation}\label{eq:tvct}
     \underset{\xx \in \mathbb{R}^n}{\text{minimize }}  \underbrace{\frac{1}{2}\|\boldsymbol{A} \xx-\bb\|_2^2}_{f(\xx)} + \underbrace{\lambda\, \mathrm{TV}_{\mathrm{r}}(\xx) + \iota_{\xx \geq 0}(\xx)}_{g(\xx)} + \underbrace{\lambda\, \mathrm{TV}_{\mathrm{c}}(\xx)}_{h(\xx)},
\end{equation}
where $\boldsymbol{A} \in \mathbb{R}^{m \times n}$ with $m=1440$, $n=2500$ is the CT system matrix from TAO's tomography example \citep{dener2021tao}, $\bb$ is the sinogram with $1\%$ additive Gaussian noise, and $\mathrm{TV}_{\mathrm{r}}$ ($\mathrm{TV}_{\mathrm{c}}$) denotes the sum of one-dimensional total variations of the rows (columns) of the image.
The weighted proximal operators are described in \cref{supp:diagvm}.
We use $\lambda=10^{-8}$.

\subsection{Overlapping Group Lasso on RCV1}\label{sec:e3}
We consider logistic regression with overlapping group lasso regularization \citep{jacob2009group}.  The groups are divided into two collections, each containing disjoint groups:
\begin{equation}\label{eq:grouplasso}
    \underset{\xx \in \mathbb{R}^n}{\text{minimize }} \underbrace{\frac{1}{n_d} \sum_{i=1}^{n_d} \log(1+ \exp (-b_i \va_i^T \xx))}_{f(\xx)}
    + \underbrace{\lambda \sum_{G \in \mathbb{G}} \| [\xx]_G \|}_{g(\xx)}
    + \underbrace{\lambda \sum_{H \in \mathbb{H}} \| [\xx]_H \|}_{h(\xx)},
\end{equation}
with groups of size 10 and stride 8, $\{1,\dots,10\}, \{9,\dots,18\}, \dots$, split alternately into $\mathbb{G}$ (odd groups) and $\mathbb{H}$ (even groups),
exactly as in \citet[Appendix~E.3]{pedregosa2018adaptive}.
We use the RCV1 text categorization dataset \citep{lewis2004rcv1}, consisting of $n_d=23{,}149$ documents with feature vectors $\va_i\in\mathbb{R}^n$, restricted to the $n=10{,}002$ features with the highest document frequency, and binary labels $b_i\in\{-1,+1\}$ ($+1$ for CCAT, $-1$ otherwise).
The weighted proximal operators are described in \cref{supp:diagvm}.
We use $\lambda=3.249\cdot10^{-5}$.

\subsection{Portfolio Optimization}\label{sec:e2}
We consider the portfolio optimization problem of \citet[Section~5.4]{davis2017three}: minimizing investment risk over the unit simplex subject to a minimum expected return,
\begin{equation}\label{eq:port}
\begin{aligned}
\underset{\xx \in \mathbb{R}^n}{\text{minimize}} \quad & \frac{1}{2}\xx^T \boldsymbol{C} \xx \\
\text{subject to} \quad & \sum_{i=1}^n \nu_i x_i \geq r, \quad \sum_{i=1}^n x_i = 1, \quad \xx \geq 0.
\end{aligned}
\end{equation}
Its three operator formulation is
\begin{equation}\label{eq:port3}
\underset{\xx \in \mathbb{R}^n}{\text{minimize }} \underbrace{\frac{1}{2} \xx^T \boldsymbol{C} \xx}_{f(\xx)}  + \underbrace{\iota_1(\xx)}_{g(\xx)} + \underbrace{\iota_2(\xx)}_{h(\xx)}.
\end{equation}
Here $\iota_1$ and $\iota_2$ are the convex indicator functions of the unit simplex and halfspace, respectively.  Their proximal operators are the corresponding metric projections.

For projection onto the unit simplex under a diagonal metric (a continuous quadratic knapsack problem), see \citet[\S3.2]{park2020variable_arxiv} and \citet{kiwiel2008breakpoint}.
For projection onto a halfspace under a diagonal metric, see \citet[Example~29.20]{Bauschke2017}.

We use daily returns weighted by market value from the 100 Fama--French portfolios formed on size and book to market (1945--2026; $15{,}897$ days with complete data for all 100 portfolios) \citep{frenchdata}.  We estimate the covariance matrix $\boldsymbol{C}$ and mean return vector $\boldsymbol{\nu}$ from these data and set $r=n^{-1}\sum_{i=1}^n\nu_i$, following \citet[Section~5.1]{yurtsever2016stochastic}.  We also use a synthetic instance with $1000$ assets, generated as in \citet[Section~5.4]{davis2017three}, with $r=0.15$.

\subsection{Synthetic Anisotropic Example}\label{sec:e4}
We investigate the performance of variable metrics on a synthetic problem with an anisotropic Hessian.

We use the overlapping group penalties of \cref{sec:e3}, with $f(\xx) = \frac{1}{2}\|\boldsymbol{A}\xx - \bb\|^2$ and $\lambda=10^{-2}$.
We generate $\boldsymbol{A} = \boldsymbol{R}\,\mathrm{diag}(\boldsymbol{d})$, where $\boldsymbol{R}\in\mathbb{R}^{m\times n}$ is obtained by normalizing
the columns of a Gaussian matrix ($n=482$, $m=674$), and the entries $d_i$ are logarithmically spaced on $[1,\sqrt{\kappa}]$.
By construction, the ratio of the largest to smallest diagonal entry of $\boldsymbol{A}^T\boldsymbol{A}$ equals the chosen $\kappa$ value.
We generate $\bb$ from a signal with group sparsity and add Gaussian noise with standard deviation $0.01$.
We test $\kappa\in\{10^1,\ldots,10^5\}$, with ten random instances for each value and orientation.  For this example, the goal is to reduce the optimality gap $P(\xx_k)-P^\star$ to $10^{-4}$ of its initial value within $10^4$ iterations.

Following the rotated benchmark comparisons of \citet{auger2009experimental}, we also examine sensitivity to coordinate orientation by replacing $\boldsymbol A$ with $\boldsymbol A\boldsymbol Q$, where $\boldsymbol Q$ is a random orthogonal matrix \citep{mezzadri2006generate}.
This preserves the Hessian eigenvalues while changing its orientation relative to the coordinate axes.  The group lasso penalties are unchanged.

\subsection{Results}

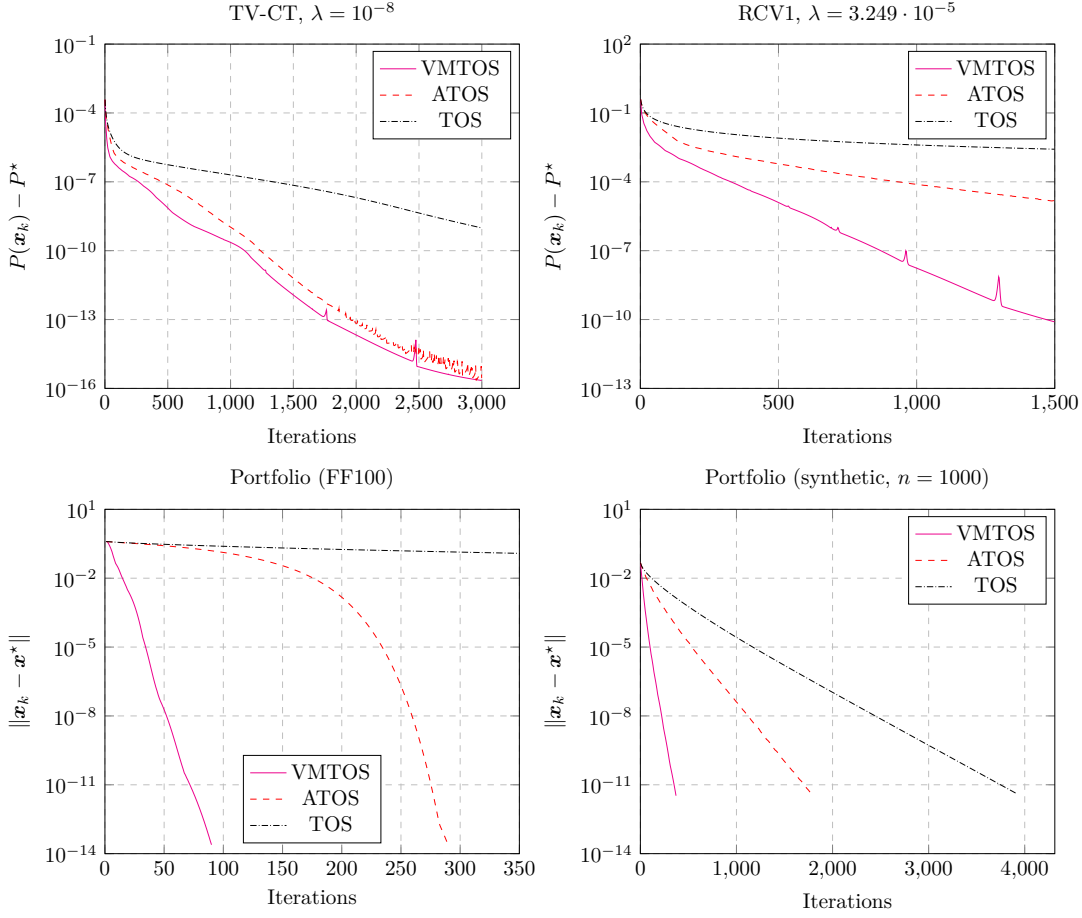
\begin{figure*}
\begin{center}
\begin{tikzpicture}[yscale=0.8,xscale=0.8]
\begin{groupplot}[group style={group size=2 by 2, horizontal sep = 2cm, vertical sep= 2cm}, ymode=log, unbounded coords=jump, xlabel={Iterations},grid=both,grid style=dashed, no markers]

    \nextgroupplot[title={TV-CT, $\lambda=10^{-8}$},ylabel={$P(\xx_k)-P^\star$},xmin=0,ymin=1e-16,ymax=1e-1]
    \addplot[magenta, each nth point=2] table[x=iter,y=y]{data/certified-references-2026-09-28/fig_e1_convergence_vmtosls.txt};
            \addlegendentry{VMTOS}
    \addplot[red, dashed, each nth point=2] table[x=iter,y=y]{data/certified-references-2026-09-28/fig_e1_convergence_atos.txt};
            \addlegendentry{ATOS}
    \addplot[black, densely dashdotted, each nth point=2] table[x=iter,y=y]{data/certified-references-2026-09-28/fig_e1_convergence_tos.txt};
            \addlegendentry{TOS}

    \nextgroupplot[title={RCV1, $\lambda=3.249\cdot10^{-5}$},ylabel={$P(\xx_k)-P^\star$},xmin=0,xmax=1500,xtick={0,500,1000,1500},ymin=1e-13,ymax=1e2]
    \addplot[magenta, each nth point=2] table[x=iter,y=y]{data/certified-references-2026-09-28/fig_e3_convergence_3.249e-05_vmtosls.txt};
            \addlegendentry{VMTOS}
    \addplot[red, dashed, each nth point=2] table[x=iter,y=y]{data/certified-references-2026-09-28/fig_e3_convergence_3.249e-05_atos.txt};
            \addlegendentry{ATOS}
    \addplot[black, densely dashdotted, each nth point=2] table[x=iter,y=y]{data/certified-references-2026-09-28/fig_e3_convergence_3.249e-05_tos.txt};
            \addlegendentry{TOS}

    \nextgroupplot[title={Portfolio (FF100)},ylabel={$\|\xx_k - \xx^\star\|$},xmin=0,xmax=350,ymin=1e-14,ymax=1e1,legend style={at={(0.5,0.03)},anchor=south}]
    \addplot[magenta] table[x=iter,y=y]{data/portfolio-targets-2026-09-29/fig_e2_convergence_ff100_vmtosls.txt};
            \addlegendentry{VMTOS}
    \addplot[red, dashed] table[x=iter,y=y]{data/portfolio-targets-2026-09-29/fig_e2_convergence_ff100_atos.txt};
            \addlegendentry{ATOS}
    \addplot[black, densely dashdotted, each nth point=10] table[x=iter,y=y]{data/portfolio-targets-2026-09-29/fig_e2_convergence_ff100_tos.txt};
            \addlegendentry{TOS}

    \nextgroupplot[title={Portfolio (synthetic, $n=1000$)},ylabel={$\|\xx_k - \xx^\star\|$},xmin=0,ymin=1e-14,ymax=1e1]
    \addplot[magenta, each nth point=10] table[x=iter,y=y]{data/portfolio-targets-2026-09-29/fig_e2_convergence_synth1000_vmtosls.txt};
            \addlegendentry{VMTOS}
    \addplot[red, dashed, each nth point=10] table[x=iter,y=y]{data/portfolio-targets-2026-09-29/fig_e2_convergence_synth1000_atos.txt};
            \addlegendentry{ATOS}
    \addplot[black, densely dashdotted, each nth point=10] table[x=iter,y=y]{data/portfolio-targets-2026-09-29/fig_e2_convergence_synth1000_tos.txt};
            \addlegendentry{TOS}

\end{groupplot}
\end{tikzpicture}
\end{center}
\caption{Convergence histories for TOS, ATOS, and VMTOS. $P^\star$ is a dual lower bound, and $\xx^\star$ is the optimal solution.}
\label{plt:plots}
\end{figure*}

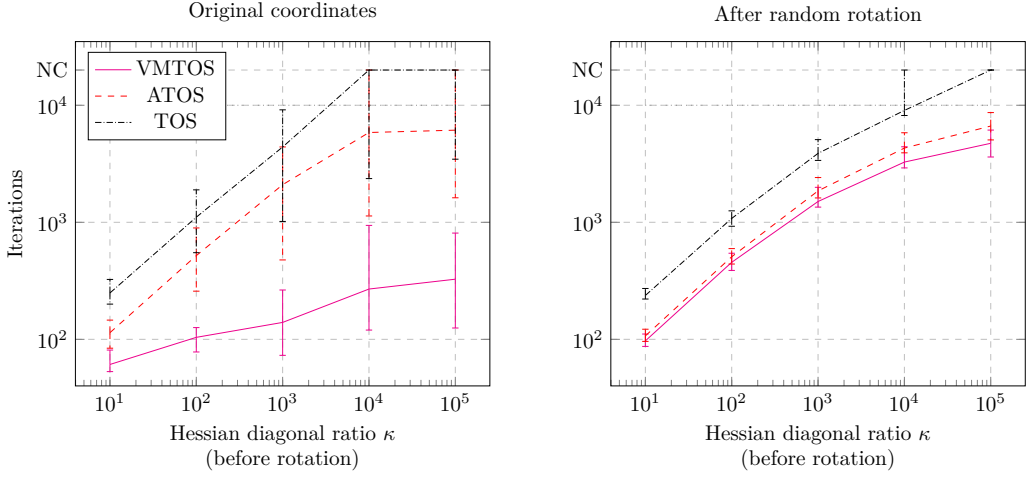
\begin{figure*}
\begin{center}
\begin{tikzpicture}[yscale=0.8,xscale=0.8]
\begin{groupplot}[group style={group size=2 by 1, horizontal sep = 2cm}, xmode=log, ymode=log, xlabel={Hessian diagonal ratio $\kappa$\\(before rotation)}, xlabel style={align=center}, ymin=40, ymax=35000, ytick={100,1000,10000,20000}, yticklabels={$10^2$,$10^3$,$10^4$,NC}, grid=major, grid style=dashed, no markers]

    \nextgroupplot[title={Original coordinates}, ylabel={Iterations}, legend pos=north west]
    \addplot[magenta, mark=*, mark size=1.5pt, error bars/.cd, y dir=both, y explicit] table[x=kappa, y=med, y error minus expr={\thisrow{med}-\thisrow{lo}}, y error plus expr={\thisrow{hi}-\thisrow{med}}]{data/certified-references-2026-09-28/e4_axis_vmtosls.txt};
            \addlegendentry{VMTOS}
    \addplot[red, dashed, mark=triangle*, mark size=1.5pt, error bars/.cd, y dir=both, y explicit] table[x=kappa, y=med, y error minus expr={\thisrow{med}-\thisrow{lo}}, y error plus expr={\thisrow{hi}-\thisrow{med}}]{data/certified-references-2026-09-28/e4_axis_atos.txt};
            \addlegendentry{ATOS}
    \addplot[black, densely dashdotted, mark=diamond*, mark size=1.5pt, error bars/.cd, y dir=both, y explicit] table[x=kappa, y=med, y error minus expr={\thisrow{med}-\thisrow{lo}}, y error plus expr={\thisrow{hi}-\thisrow{med}}]{data/certified-references-2026-09-28/e4_axis_tos.txt};
            \addlegendentry{TOS}
    \addplot[gray, dotted, domain=1e1:1e5] {10000};

    \nextgroupplot[title={After random rotation}]
    \addplot[magenta, mark=*, mark size=1.5pt, error bars/.cd, y dir=both, y explicit] table[x=kappa, y=med, y error minus expr={\thisrow{med}-\thisrow{lo}}, y error plus expr={\thisrow{hi}-\thisrow{med}}]{data/certified-references-2026-09-28/e4_rot_vmtosls.txt};
    \addplot[red, dashed, mark=triangle*, mark size=1.5pt, error bars/.cd, y dir=both, y explicit] table[x=kappa, y=med, y error minus expr={\thisrow{med}-\thisrow{lo}}, y error plus expr={\thisrow{hi}-\thisrow{med}}]{data/certified-references-2026-09-28/e4_rot_atos.txt};
    \addplot[black, densely dashdotted, mark=diamond*, mark size=1.5pt, error bars/.cd, y dir=both, y explicit] table[x=kappa, y=med, y error minus expr={\thisrow{med}-\thisrow{lo}}, y error plus expr={\thisrow{hi}-\thisrow{med}}]{data/certified-references-2026-09-28/e4_rot_tos.txt};
    \addplot[gray, dotted, domain=1e1:1e5] {10000};

\end{groupplot}
\end{tikzpicture}
\end{center}
\caption{Synthetic anisotropic example (\cref{sec:e4}): The horizontal axis gives the ratio of the largest to smallest Hessian diagonal entry before rotation. Points show medians over ten random instances, and bars show their range. NC denotes a run that did not converge within $10^4$ iterations. If at least five of ten runs are NC, the median is also shown as NC.}
\label{plt:e4}
\end{figure*}

\begin{figure*}
\begin{center}
\begin{tikzpicture}[yscale=0.8,xscale=0.8]
\begin{groupplot}[group style={group size=2 by 1, horizontal sep = 2.2cm}, xlabel={Iterations}, grid=both, grid style=dashed, no markers]

    \nextgroupplot[title={TV-CT, $\lambda=10^{-8}$}, ylabel={Cumulative backtracking steps}, legend pos=north west]
    \addplot[magenta, each nth point=2] table[x=iter,y=cumback]{data/three-solvers-2026-09-21/mech_backtracks_vmtosls.txt};
            \addlegendentry{VMTOS}
    \addplot[red, dashed, each nth point=2] table[x=iter,y=cumback]{data/three-solvers-2026-09-21/mech_backtracks_atos.txt};
            \addlegendentry{ATOS}

    \nextgroupplot[title={RCV1, $\lambda=3.249\cdot10^{-5}$}, ylabel={Cumulative backtracking steps}, legend pos=north west]
    \addplot[magenta, each nth point=2] table[x=iter,y=cumback]{data/three-solvers-2026-09-21/mech_backtracks_rcv1_vmtosls.txt};
            \addlegendentry{VMTOS}
    \addplot[red, dashed, each nth point=2] table[x=iter,y=cumback]{data/three-solvers-2026-09-21/mech_backtracks_rcv1_atos.txt};
            \addlegendentry{ATOS}

\end{groupplot}
\end{tikzpicture}
\end{center}
\caption{Cumulative backtracking steps for tomography (left) and RCV1 (right).  On tomography, VMTOS backtracks $22$ times and ATOS $164$ times over $3000$ iterations.  On RCV1, the counts are $21$ and $74$, respectively, over $1500$ iterations.}
\label{plt:mech}
\end{figure*}
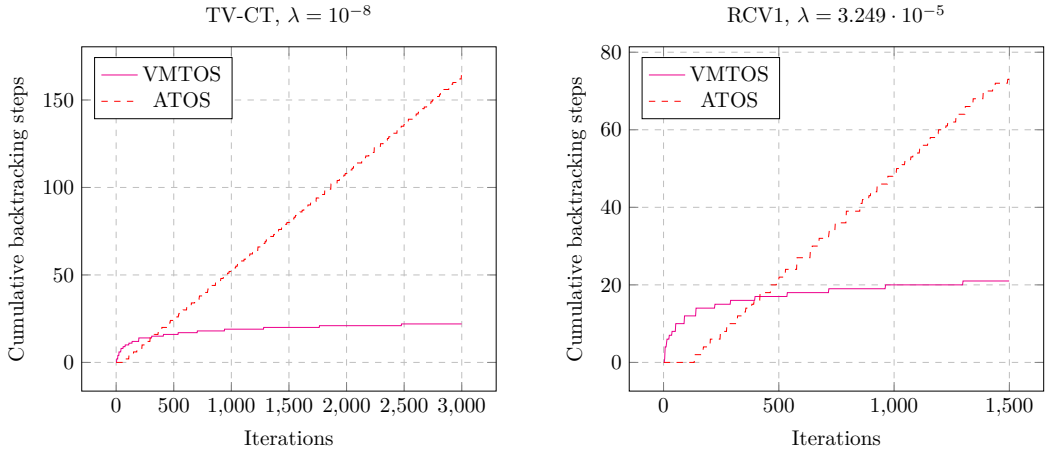

\cref{plt:plots} compares convergence on the tomography, RCV1, and two portfolio problems.
In all four instances, VMTOS achieves comparable accuracy in fewer iterations than TOS and ATOS.

\cref{plt:e4} shows that, in the original coordinates, VMTOS reached the target in substantially fewer iterations than TOS and ATOS, consistent with its diagonal metric adapting to differences in curvature along the coordinate axes.
At $\kappa=10^3$, the median counts were $139.5$ iterations for VMTOS and $2095$ for ATOS.
After rotation, the corresponding medians were $1505.5$ and $1855$ iterations.
Rotation preserves the Hessian eigenvalues but mixes the coordinate directions, making the anisotropy harder to capture with a diagonal metric.

\Cref{plt:mech} shows that VMTOS requires fewer backtracking steps than ATOS on both the tomography and RCV1 examples.

\Cref{plt:oracle} shows the same error histories as \cref{plt:plots} against cumulative gradient and proximal evaluations.  Counts include rejected trials and the additional gradients used to compute residuals, assigning unit cost to each evaluation.
VMTOS also achieves comparable accuracy with fewer gradient and proximal evaluations than TOS and ATOS.

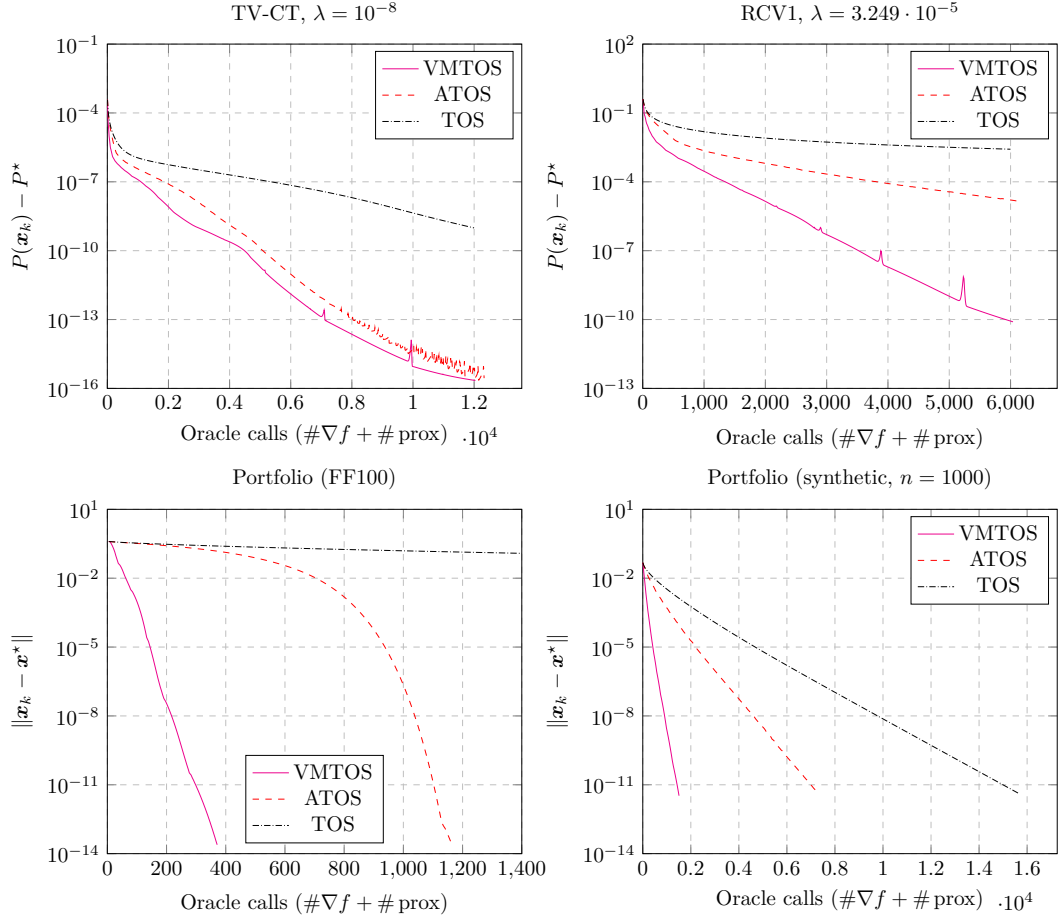
\begin{figure*}
\begin{center}
\begin{tikzpicture}[yscale=0.8,xscale=0.8]
\begin{groupplot}[group style={group size=2 by 2, horizontal sep = 2cm, vertical sep= 2cm}, ymode=log, unbounded coords=jump, xlabel={Oracle calls ($\#\nabla f + \#\prox$)},grid=both,grid style=dashed, no markers]

    \nextgroupplot[title={TV-CT, $\lambda=10^{-8}$},ylabel={$P(\xx_k)-P^\star$},xmin=0,ymin=1e-16,ymax=1e-1]
    \addplot[magenta, each nth point=2] table[x=oracle,y=y]{data/certified-references-2026-09-28/fig_e1_convergence_vmtosls.txt};
            \addlegendentry{VMTOS}
    \addplot[red, dashed, each nth point=2] table[x=oracle,y=y]{data/certified-references-2026-09-28/fig_e1_convergence_atos.txt};
            \addlegendentry{ATOS}
    \addplot[black, densely dashdotted, each nth point=2] table[x=oracle,y=y]{data/certified-references-2026-09-28/fig_e1_convergence_tos.txt};
            \addlegendentry{TOS}

    \nextgroupplot[title={RCV1, $\lambda=3.249\cdot10^{-5}$},ylabel={$P(\xx_k)-P^\star$},xmin=0,ymin=1e-13,ymax=1e2]
    \addplot[magenta, each nth point=2] table[x=oracle,y=y]{data/certified-references-2026-09-28/fig_e3_convergence_3.249e-05_vmtosls.txt};
            \addlegendentry{VMTOS}
    \addplot[red, dashed, each nth point=2] table[x=oracle,y=y]{data/certified-references-2026-09-28/fig_e3_convergence_3.249e-05_atos.txt};
            \addlegendentry{ATOS}
    \addplot[black, densely dashdotted, each nth point=2] table[x=oracle,y=y]{data/certified-references-2026-09-28/fig_e3_convergence_3.249e-05_tos.txt};
            \addlegendentry{TOS}

    \nextgroupplot[title={Portfolio (FF100)},ylabel={$\|\xx_k - \xx^\star\|$},xmin=0,xmax=1400,ymin=1e-14,ymax=1e1,legend style={at={(0.5,0.03)},anchor=south}]
    \addplot[magenta] table[x=oracle,y=y]{data/portfolio-targets-2026-09-29/fig_e2_convergence_ff100_vmtosls.txt};
            \addlegendentry{VMTOS}
    \addplot[red, dashed] table[x=oracle,y=y]{data/portfolio-targets-2026-09-29/fig_e2_convergence_ff100_atos.txt};
            \addlegendentry{ATOS}
    \addplot[black, densely dashdotted, each nth point=10] table[x=oracle,y=y]{data/portfolio-targets-2026-09-29/fig_e2_convergence_ff100_tos.txt};
            \addlegendentry{TOS}

    \nextgroupplot[title={Portfolio (synthetic, $n=1000$)},ylabel={$\|\xx_k - \xx^\star\|$},xmin=0,ymin=1e-14,ymax=1e1]
    \addplot[magenta, each nth point=10] table[x=oracle,y=y]{data/portfolio-targets-2026-09-29/fig_e2_convergence_synth1000_vmtosls.txt};
            \addlegendentry{VMTOS}
    \addplot[red, dashed, each nth point=10] table[x=oracle,y=y]{data/portfolio-targets-2026-09-29/fig_e2_convergence_synth1000_atos.txt};
            \addlegendentry{ATOS}
    \addplot[black, densely dashdotted, each nth point=10] table[x=oracle,y=y]{data/portfolio-targets-2026-09-29/fig_e2_convergence_synth1000_tos.txt};
            \addlegendentry{TOS}

\end{groupplot}
\end{tikzpicture}
\end{center}
\caption{The convergence histories of \cref{plt:plots} plotted against cumulative gradient and proximal evaluations, including rejected line-search trials and the additional gradients used to compute residuals.}
\label{plt:oracle}
\end{figure*}

\section{Conclusion}
We presented VMTOS, a variable-metric extension of three operator splitting, and established a $\tilde{O}(k^{-1/2})$ bound on the smallest saddle-point residual among the first $k$ primal-dual iterates by controlling metric changes.
Under strong convexity of $f$ and $h^*$ and uniform metric bounds, the constrained Barzilai--Borwein rule with fixed regularization improves the iterate bound from $O((\log k)^\rho\beta^k)$ to $O(\beta^k)$.
Our practical diagonal BB rule guarantees the required uniform metric bounds when $f$ is strongly convex, yielding R-linear convergence when $h^*$ is also strongly convex.

On the tested problems, VMTOS achieved comparable accuracy with fewer iterations and fewer gradient and proximal evaluations than TOS and ATOS.
The synthetic experiments showed that the benefit of diagonal adaptation depends on coordinate orientation.
Future work could explore metrics consisting of a diagonal matrix plus a rank-one correction and extend the variable-metric framework to sums of more than three convex functions \citep{ryu2020uniqueness,malitsky2023resolvent}.

\bibliography{vmtos_tmlr}
\bibliographystyle{tmlr}

\newpage
\appendix
\onecolumn
\section{Proofs}

\subsection{Proof of \texorpdfstring{\cref{prop:backtrack}}{backtrack}}
\label{proof:backtrack}

\begin{proof}
Suppose $L>0$, without loss of generality.  An $L$-smooth convex function satisfies
$f(\xx) \leq f(\zz) + \langle \nabla f(\zz), \xx - \zz \rangle + \frac{L}{2} \|\xx - \zz\|^2$ for all $\xx, \zz \in \mathbb{R}^n.$
Therefore sufficient decrease is achieved if $L \|\xx_{k+1} - \zz_k\|^2 \leq \sigma \|\xx_{k+1} - \zz_k\|_{\MM_k}^2$.  After backtracking $\ell$ times on line \ref{line:strengthen-metric}, $\MM_k = \tau^\ell \widehat\MM_k$, so \[\frac{\sigma}{2}\|\xx_{k+1} - \zz_k \|_{\MM_k}^2 \geq \frac{\sigma}{2} \tau^\ell \lambda_{\min}(\widehat\MM_k)\|\xx_{k+1} - \zz_k\|^2,\]
where $\lambda_{\min}(\widehat\MM_k) > 0$ is the smallest eigenvalue of $\widehat\MM_k$.
Thus sufficient decrease holds once $\sigma \tau^\ell \lambda_{\min}(\widehat\MM_k) \geq L$, so the number of backtracking steps satisfies
\begin{equation}\label{eq:backtrack-bound}
\ell \leq \max\{0,\lceil\log_{\tau} (L / (\sigma \lambda_{\min}(\widehat\MM_k))) \rceil\}.
\end{equation}
\end{proof}

\subsection{Proof of \texorpdfstring{\cref{lem:keyrecursive}}{keyrecursive} (recursive inequality)}
\label{proof:keyrecursive}

\begin{proof}
We follow the proof of Lemma 4 in
\citet{pedregosa2018adaptive}.
Let $(\xx,\uu)\in\Omega$ be given.
The sufficient decrease condition on line \ref{line:sufficient-decrease} can be written as
\[
    f(\xx_{k+1}) - f(\zz_k) \leq \langle \nabla f(\zz_k), \xx_{k+1}-\zz_k \rangle + \frac{1}{2}\|\xx_{k+1}-\zz_k\|_{\MM_k}^2 - \delta_k,
\]
where $\delta_k = \frac{(1-\sigma)}{2} \|\xx_{k+1} - \zz_k\|_{\MM_k}^2$.
Combining this with the definition of $\mu_f$-strong convexity, we have
\begin{equation}\label{56}
f(\xx_{k+1})-f(\xx) \leq \langle \nabla f(\zz_k), \xx_{k+1}-\xx \rangle + \frac{1}{2}\|\zz_k-\xx_{k+1}\|_{\MM_k}^2 - \frac{\mu_f}{2}\|\zz_k-\xx\|_2^2 - \delta_k. 
\end{equation}
Equation \eqref{eq:tos-iter-prox-g} implies the inclusion
$\MM_k (\zz_k - \MM_k^{-1}(\uu_k + \nabla f(\zz_k)) - \xx_{k+1})  \in \partial g(\xx_{k+1})$,
and so
\begin{align}
    g(\xx_{k+1})-g(\xx) &\leq  \langle \MM_k \zz_k - (\uu_k + \nabla f(\zz_k)) - \MM_k \xx_{k+1}, \xx_{k+1} - \xx \rangle \\
    &= \langle \MM_k (\zz_k - \xx_{k+1}),\xx_{k+1}-\xx \rangle - \langle \uu_k + \nabla f(\zz_k), \xx_{k+1}-\xx \rangle. \label{59}
\end{align}
Adding \eqref{56} and \eqref{59}, we obtain
\begin{multline}
    f(\xx_{k+1})+ g(\xx_{k+1})-f(\xx)-g(\xx) \\ \leq
     \langle \MM_k (\zz_k - \xx_{k+1}),\xx_{k+1}-\xx \rangle 
    + \frac{1}{2}\|\zz_k-\xx_{k+1}\|_{\MM_k}^2
    - \langle \uu_k, \xx_{k+1}-\xx \rangle 
    - \frac{\mu_f}{2}\|\zz_k - \xx\|_2^2 - \delta_k.
\end{multline}
As in \citet{pedregosa2018adaptive}, we use the above inequality to bound $\LL(\xx_{k+1},\uu_{k+1})-\LL(\xx,\uu_{k+1})$.  Noting the primal-dual relationship $\uu_{k+1}-\uu_k=\MM_k(\xx_{k+1}-\zz_{k+1})$, and that the cosine identity holds for any inner-product norm, $2 \langle\MM \va,\bb \rangle = \|\va+\bb\|_{\MM}^2 - \|\va\|_{\MM}^2 - \|\bb\|_{\MM}^2$, we have the following:
\begin{align}
&\LL(\xx_{k+1},\uu_{k+1})- \LL(\xx,\uu_{k+1}) = f(\xx_{k+1})+ g(\xx_{k+1})-f(\xx)-g(\xx)  + \langle \uu_{k+1} , \xx_{k+1}-\xx \rangle \\
\leq{} &\langle \MM_k (\zz_k - \xx_{k+1}), \xx_{k+1}-\xx \rangle + \frac{1}{2}\|\zz_k-\xx_{k+1}\|_{\MM_k}^2 + \langle \uu_{k+1}-\uu_k ,\xx_{k+1}-\xx\rangle 
- \frac{\mu_f}{2}\|\zz_k-\xx\|_2^2 - \delta_k\\
={} &\frac{1}{2}\|\zz_k-\xx\|_{\MM_k}^2-\frac{1}{2}\| \zz_{k+1}-\xx\|_{\MM_k}^2 + \frac{1}{2}\|\zz_{k+1}-\xx_{k+1}\|_{\MM_k}^2 - \frac{\mu_f}{2}\|\zz_k-\xx\|_2^2 -\delta_k. \label{64}
\end{align}

Turning to $h^*$, equations \eqref{eq:tos-iter-prox-h} and \eqref{eq:tos-iter-dual-update} together with the subgradient inverse relationship \eqref{eq:subdiff-inverse} imply the inclusion
$\zz_{k+1} \in \partial h^*(\uu_{k+1})$, and so with $\mu_h$-strong convexity we have
\begin{equation}
    h^*(\uu_{k+1}) - h^*(\uu) \leq \langle \zz_{k+1}, \uu_{k+1}-\uu \rangle - \frac{\mu_h}{2}\|\uu_{k+1}-\uu\|_2^2.
\end{equation}
With this we bound $\LL(\xx_{k+1},\uu) - \LL(\xx_{k+1},\uu_{k+1})$:
\begin{align}
&\LL(\xx_{k+1},\uu) - \LL(\xx_{k+1},\uu_{k+1}) = \langle \uu - \uu_{k+1}, \xx_{k+1} \rangle - h^*(\uu) + h^*(\uu_{k+1}) \\
\leq{} &\langle \uu_{k+1}-\uu, \zz_{k+1}-\xx_{k+1} \rangle - \frac{\mu_h}{2} \|\uu_{k+1}-\uu\|_2^2\\
={} &\langle \uu_{k+1}-\uu , \MM_k^{-1} (\uu_k - \uu_{k+1}) \rangle - \frac{\mu_h}{2} \|\uu_{k+1}-\uu\|_2^2\\
={} &\frac{1}{2}\|\uu_k -\uu\|_{\MM_k^{-1}}^2 - \frac{1}{2}\|\uu_k-\uu_{k+1}\|_{\MM_k^{-1}}^2 - \frac{1}{2}\|\uu_{k+1}-\uu\|_{\MM_k^{-1}}^2 - \frac{\mu_h}{2} \|\uu_{k+1}-\uu\|_2^2 \\
={} &\frac{1}{2}\|\uu_k -\uu\|_{\MM_k^{-1}}^2 - \frac{1}{2}\|\xx_{k+1} - \zz_{k+1}\|_{\MM_k}^2 - \frac{1}{2}\|\uu_{k+1}-\uu\|_{\MM_k^{-1}}^2 - \frac{\mu_h}{2} \|\uu_{k+1}-\uu\|_2^2.
\label{71}
\end{align}
Adding together \eqref{64} and \eqref{71}, we have the desired bound of the lemma, 
\begin{multline}
2(\LL(\xx_{k+1},\uu)-\mathcal{L}(\xx,\uu_{k+1})) \\ \leq
\|\zz_k-\xx\|_{\MM_k}^2 - \mu_f \|\zz_k-\xx\|_2^2 - \|\zz_{k+1}-\xx\|_{\MM_k}^2 - (1 - \sigma) \|\zz_k - \xx_{k+1}\|_{\MM_k}^2  \\
+\|\uu_k-\uu\|_{\MM_{k}^{-1}}^2 - \mu_h \|\uu_{k+1} - \uu\|_2^2 - \|\uu_{k+1} - \uu\|_{\MM_k^{-1}}^2.
\end{multline}
\end{proof}

\subsection{Proof of \texorpdfstring{\cref{thm:metric-distance} ($\Delta_{0,k} \in O((\log k)^{2\rho})$}{Dlog})}
\label{proof:metric-distance}

We start with a lemma that bounds $\prod_{j=1}^k d_j$.

\begin{lemma}\label{lem:d-prod}
Given $\rho \geq 0$,
\[\prod_{j=1}^k \left(1 + \frac{\rho}{j(1 + \log j)}\right) \in O((\log k)^\rho).\]
\end{lemma}

\begin{proof}
Taking the logarithm of the product and using the bound $\log(1+y) \leq y$, we see
\begin{equation}
\log\left\{\prod_{j=1}^k \left(1 + \frac{\rho}{j(1 + \log j)}\right)\right\}
=
\sum_{j=1}^k \log\left(1+ \frac{\rho}{j(1 + \log j)}\right)
\leq
\rho \sum_{j=1}^k \frac{1}{j(1 + \log j)}. \label{eq:logprod}
\end{equation}
The summand on the right-hand side of \eqref{eq:logprod} decreases monotonically to zero, so there is some constant $\tilde\gamma$ such
that
\begin{equation}
\rho \sum_{j=1}^k \frac{1}{j(1 + \log j)}
\leq
\rho\left(\tilde \gamma + \int_1^{k} \frac{1}{x(1 + \log x)}\ dx\right)
=
\rho(\tilde\gamma + \log(1 + \log k)).
\end{equation}
Therefore
\begin{equation}
\prod_{j=1}^k \left(1 + \frac{\rho}{j(1 + \log j)}\right) \leq \exp(\tilde \gamma \rho)(1 + \log k)^\rho \in O((1 + \log k)^\rho).
\end{equation}
\end{proof}

Now the main proof.

\begin{proof}
Let $\ell_k$ be the number of backtracking steps taken at iteration $k$ on line \ref{line:strengthen-metric}.  By \cref{prop:backtrack}, this number is finite, and
$\MM_k = \tau^{\ell_k} \widehat\MM_k$ and $\Delta_{k-1,k} = \Delta(\MM_{k-1},\MM_k) \leq \tau^{\ell_k} \Delta(\MM_{k-1},\widehat\MM_k)$.
Therefore
\begin{equation}\label{eq:delta-bound}
    \Delta_{0,k} = \prod_{j=1}^{k} \Delta_{j-1,j} \leq \prod_{j=1}^k \tau^{\ell_j} \Delta(\MM_{j-1},\widehat\MM_j) \leq \left(\prod_{j=1}^k \tau^{\ell_j}\right)\left(\prod_{j=1}^k d_j\right).
\end{equation}
By \cref{lem:d-prod}, the second term on the right is $O((\log k)^\rho)$.  We now bound the first term.

Using the bound \eqref{eq:backtrack-bound} from \cref{prop:backtrack}, we have
\begin{equation}\label{eq:bound-tsk}
\tau^{\ell_k} \leq \max\left\{1, \frac{\tau L}{\sigma \lambda_{\min} (\widehat\MM_k)}\right\}.
\end{equation}
We can relate $\lambda_{\min}(\widehat\MM_k)$ to $\lambda_{\min}(\MM_{k-1})$,
\[
\lambda_{\min}(\widehat\MM_k) = \inf_{\xx: \|\xx\| = 1} \|\xx\|_{\widehat\MM_k}^2
\geq
\inf_{\xx: \|\xx\| = 1} \Delta(\MM_{k-1},\widehat\MM_k)^{-1} \|\xx\|_{\MM_{k-1}}^2
\geq
d_k^{-1}
\lambda_{\min}(\MM_{k-1}),
\]
which implies $\lambda_{\min}(\MM_k) \geq \frac{\tau^{\ell_k}}{d_k}\lambda_{\min}(\MM_{k-1})$.
Applying this bound recursively to \eqref{eq:bound-tsk}, we get
\[
\tau^{\ell_k} \leq \max\left\{1, \frac{\tau L \prod_{j=1}^k d_j}{\sigma \lambda_{\min} (\MM_0)\prod_{j=1}^{k-1} \tau^{\ell_j}}\right\},
\]
and so
\[
\prod_{j=1}^k\tau^{\ell_j} \leq \max\left\{\prod_{j=1}^{k-1} \tau^{\ell_j}, \frac{\tau L }{\sigma \lambda_{\min} (\MM_0)}\prod_{j=1}^k d_j\right\}.
\]
This recursively simplifies to
\[
\prod_{j=1}^k\tau^{\ell_j} \leq
\max\left\{1, \frac{\tau L }{\sigma \lambda_{\min} (\MM_0)}\prod_{j=1}^k d_j \right\} \in O((\log k)^\rho).
\]
Therefore the product in the bound on the right-hand side of \eqref{eq:delta-bound} is $O((\log k)^{2\rho})$.

\end{proof}

\subsection{Proof of \texorpdfstring{\cref{thm:pointwise-convergence}}{residual convergence} (pointwise convergence of a subsequence)}
\label{proof:pointwise-convergence}

For this proof, we need to relate the residual $r(\xx_{k+1},\uu_{k+1})$ to the norm of $(\xx_{k+1}- \zz_k)$.

\begin{lemma}\label{lem:residual-bound}
Let assumptions (\ref{eq:proper},\ref{eq:L-lipschitz},\ref{eq:u0-in-partial-h-z0}) hold.  Then for each iteration $k\geq 0$ of \vmtos
\begin{equation}\label{eq:r-xz-bound}
r(\xx_{k+1},\uu_{k+1}) 
\leq
\sqrt{C_r\Delta_{0,k}} \| \xx_{k+1} - \zz_k \|_{\MM_k},
\end{equation}
where $C_r := \max\{2\|\MM_0\|,\|\MM_0^{-1}\|\} + 2L^2\|\MM_0^{-1}\|$.
\end{lemma}

\begin{proof}
\Cref{eq:tos-iter-prox-g,eq:tos-iter-prox-h,eq:tos-iter-dual-update} imply the inclusions
\begin{align}
\MM_k(\zz_k - \zz_{k+1}) + \nabla f(\xx_{k+1}) - \nabla f(\zz_k) &\in \nabla f(\xx_{k+1}) + \partial g(\xx_{k+1}) + \uu_{k+1} = \partial_{\xx}\LL(\xx_{k+1},\uu_{k+1}), \\
\MM_k^{-1}(\uu_k - \uu_{k+1}) &\in \partial h^*(\uu_{k+1}) - \xx_{k+1} = \partial_{\uu}(-\LL(\xx_{k+1},\uu_{k+1})).
\end{align}
Therefore the norms of the included vectors are a bound on the residual,
\[
\begin{aligned}
r(\xx_{k+1},\uu_{k+1})^2 &\leq
\| \MM_k(\zz_k - \zz_{k+1}) + \nabla f(\xx_{k+1}) - \nabla f(\zz_k)\|^2 + \|\MM_k^{-1}(\uu_k - \uu_{k+1})\|^2.
\end{aligned}
\]
Using the inequality $\|\va + \bb\|^2 \leq 2 (\|\va\|^2 + \|\bb\|^2)$, we separate the contribution of gradient terms and then bound them with the Lipschitz constant,
\[
\begin{aligned}
r(\xx_{k+1},\uu_{k+1})^2
&\leq 2(\|\MM_k(\zz_k - \zz_{k+1})\|^2 + \|\nabla f(\xx_{k+1}) - \nabla f(\zz_k)\|^2) + \|\MM_k^{-1}(\uu_k - \uu_{k+1})\|^2 \\
&\leq 2(\|\MM_k(\zz_k - \zz_{k+1})\|^2 + L^2\|\xx_{k+1} - \zz_k\|^2) + \|\MM_k^{-1}(\uu_k - \uu_{k+1})\|^2.
\end{aligned}
\]
We apply the primal-dual relationship $\uu_{k+1} - \uu_k = \MM_k(\xx_{k+1} - \zz_{k+1})$ and the inequality $\|\MM \va \|^2 \leq \|\MM\| \|\va\|_{\MM}^2$ to see
\[
\begin{aligned}
r(\xx_{k+1},\uu_{k+1})^2
&\leq 2(\|\MM_k\|\|\zz_k - \zz_{k+1}\|_{\MM_k}^2 + L^2\|\MM_k^{-1}\|\|\xx_{k+1} - \zz_k\|_{\MM_k}^2) + \|\MM_k^{-1}\|\|\xx_{k+1} - \zz_{k+1}\|_{\MM_k}^2.
\end{aligned}
\]
The definition of $\Delta(\MM,\NN)$ in \eqref{eq:big-delta-def} implies that
$\Delta(\MM,\NN) = \Delta(\MM^{-1},\NN^{-1})$ and that $\|\NN\| \leq \Delta(\MM,\NN)\|\MM\|$.  Therefore we can bound the leading coefficients by terms in $\MM_0$,
\begin{equation}\label{eq:bound-r}
r(\xx_{k+1},\uu_{k+1})^2
\leq \Delta_{0,k}(2\|\MM_0\|\|\zz_k - \zz_{k+1}\|_{\MM_k}^2 + 2L^2\|\MM_0^{-1}\|\|\xx_{k+1} - \zz_k\|_{\MM_k}^2 + \|\MM_0^{-1}\|\|\xx_{k+1} - \zz_{k+1}\|_{\MM_k}^2).
\end{equation}
By the cosine identity and the primal-dual relationship
\[
\begin{aligned}
\|\xx_{k+1} - \zz_{k+1}\|_{\MM_k}^2 + \|\zz_k - \zz_{k+1}\|_{\MM_k}^2
&=\|\xx_{k+1} - \zz_k\|_{\MM_k}^2 - 2\langle \MM_k(\xx_{k+1} - \zz_{k+1}), \zz_{k+1} - \zz_k \rangle
\\
&=\|\xx_{k+1} - \zz_k\|_{\MM_k}^2 -  2\langle \uu_{k+1} - \uu_k, \zz_{k+1} - \zz_k \rangle.
\end{aligned}
\]
Because the condition $\uu_k \in \partial h(\zz_k)$ holds for every $k \geq 1$ as a consequence of \eqref{eq:tos-iter-prox-h} and \eqref{eq:tos-iter-dual-update}, and for $k=0$ by assumption \eqref{eq:u0-in-partial-h-z0}, we have
$\langle \uu_{k+1} - \uu_k, \zz_{k+1} - \zz_k \rangle \geq 0$ by \eqref{eq:monotone}, so
\begin{equation}\label{eq:bound-xz}
\|\xx_{k+1} - \zz_{k+1}\|_{\MM_k}^2 + \|\zz_k - \zz_{k+1}\|_{\MM_k}^2
\leq \|\xx_{k+1} - \zz_k\|_{\MM_k}^2.
\end{equation}
Combining \eqref{eq:bound-r} and \eqref{eq:bound-xz} gives us our final bound,
\[
r(\xx_{k+1},\uu_{k+1})^2
\leq \Delta_{0,k}
(\max\{2\|\MM_0\|,\|\MM_0^{-1}\|\} + 2L^2\|\MM_0^{-1}\|)\|\xx_{k+1} - \zz_k\|_{\MM_k}^2.
\]
\end{proof}

Now the main proof.

\begin{proof}
Let $(\xx^\star, \uu^\star)\in S(\LL)$ be chosen.  By \cref{cor:keyrecursive-solution},
\[
    \|\zz_{k+1} - \xx^\star\|_{\MM_k}^2 + \|\uu_{k+1} - \uu^\star\|_{\MM_k^{-1}}^2 \leq
    \|\zz_k - \xx^\star\|_{\MM_k}^2 + \|\uu_k - \uu^\star\|_{\MM_k^{-1}}^2 - (1 - \sigma)\|\xx_{k+1} - \zz_k\|_{\MM_k}^2.
\]
By definition \eqref{eq:delta-def}, for $k \geq 1$,
\[
    \|\zz_{k+1} - \xx^\star\|_{\MM_k}^2 + \|\uu_{k+1} - \uu^\star\|_{\MM_k^{-1}}^2 \leq
    \Delta_{k-1,k}(\|\zz_k - \xx^\star\|_{\MM_{k-1}}^2 + \|\uu_k - \uu^\star\|_{\MM_{k-1}^{-1}}^2) - (1 - \sigma)\|\xx_{k+1} - \zz_k\|_{\MM_k}^2.
\]
Applying this inequality recursively, we get
\[
    \|\zz_{k+1} - \xx^\star\|_{\MM_k}^2 + \|\uu_{k+1} - \uu^\star\|_{\MM_k^{-1}}^2 \leq
    \Delta_{0,k}(\|\zz_0 - \xx^\star\|_{\MM_{0}}^2 + \|\uu_0 - \uu^\star\|_{\MM_{0}^{-1}}^2) - (1 - \sigma)\sum_{j=0}^k\|\xx_{j+1} - \zz_j\|_{\MM_j}^2,
\]
which implies
\[
(1 - \sigma) \sum_{j=0}^k \|\xx_{j+1} - \zz_j\|_{\MM_j}^2 \leq \Delta_{0,k}(\|\zz_0 - \xx^\star\|_{\MM_{0}}^2 + \|\uu_0 - \uu^\star\|_{\MM_0^{-1}}^2).
\]
Bounding the left-hand side using the smallest element in the series, we have
\[
\min_{0\leq j\leq k} \|\xx_{j+1} - \zz_j\|_{\MM_j}^2 \leq \frac{1}{k+1} \frac{\Delta_{0,k}}{1-\sigma}(\|\zz_0 - \xx^\star\|_{\MM_{0}}^2 + \|\uu_0 - \uu^\star\|_{\MM_0^{-1}}^2).
\]
Combining this with \eqref{eq:r-xz-bound}, we get the final bound
\[
  \min_{0\leq j\leq k} r(\xx_{j+1}, \uu_{j+1})^2
  \leq \Delta_{0,k}^2\frac{C_r}{1-\sigma}\frac{1}{k+1}(\|\zz_0 - \xx^\star\|_{\MM_0}^2 +
  \|\uu_0 - \uu^\star\|_{\MM_0^{-1}}^2).
\]
\end{proof}

\subsection{Proof of \texorpdfstring{\cref{thm:linear-convergence}}{linearconv} (convergence with strong convexity)}\label{proof:linear-convergence}

(This proof closely follows the linear convergence proof in \citet{pedregosa2018adaptive}.)

\begin{proof}
We first note that assumption \eqref{eq:metric-double-bounds} implies the following inequalities:
\begin{align}
c \|\vv\|_{\MM_k^{-1}}^2 \leq c \|\MM_k^{-1}\|_2 \|\vv\|_2^2 \leq \|\vv\|_2^2,\label{eq:muf-bound}\\
C^{-1}\|\yy\|_{\MM_k}^2 \leq C^{-1}\|\MM_k\|_2 \|\yy\|_2^2 \leq \|\yy\|_2^2.\label{eq:L-bound}
\end{align}

With \eqref{eq:solutions-exist}, the strong convexity of $f$ and $h^*$ implies the saddle point is unique, $S(\LL) = \{(\xx^\star, \uu^\star)\}$.
When \cref{lem:keyrecursive} is applied to $(\xx^\star,\uu^\star)$ we obtain the inequality
\begin{equation}
\begin{split}
     2(\mathcal{L}(\xx_{k+1},\uu^\star) - \mathcal{L}(\xx^\star,\uu_{k+1})) + &
     \|\zz_{k+1}-\xx^\star\|_{\MM_{k}}^2 + \|\uu_{k+1}-\uu^\star\|_{\MM_k^{-1}}^2 + \mu_h\|\uu_{k+1}-\uu^\star\|_2^2 \leq \\
     &\|\zz_k -\xx^\star \|_{\MM_k}^2 + \|\uu_k-\uu^\star\|_{\MM_{k}^{-1}}^2 - \mu_f\|\zz_k-\xx^\star\|_2^2.
\end{split}     
\end{equation}
The saddle point gap is non-negative \citep[\S 4.1]{hiriart2013convex}, so that term can be dropped, hence
\begin{equation}\label{theineq}
     \|\zz_{k+1}-\xx^\star\|_{\MM_{k}}^2 + \|\uu_{k+1}-\uu^\star\|_{\MM_k^{-1}}^2 +  \mu_h\|\uu_{k+1}-\uu^\star\|_2^2 \leq
     \|\zz_k -\xx^\star\|_{\MM_k}^2 + \|\uu_k-\uu^\star\|_{\MM_{k}^{-1}}^2   - \mu_f\|\zz_k-\xx^\star\|_2^2.
\end{equation}
The inequality \eqref{eq:muf-bound} implies
\begin{align}
(1+ c \mu_h) \|\uu_{k+1}-\uu^\star\|_{\MM_k^{-1}}^2 &= \|\uu_{k+1}-\uu^\star\|_{\MM_k^{-1}}^2 + c\mu_h\|\uu_{k+1}-\uu^\star\|_{\MM_k^{-1}}^2 \\
 &\leq \|\uu_{k+1}-\uu^\star\|_{\MM_k^{-1}}^2 + \mu_h \|\uu_{k+1}-\uu^\star\|_2^2.
\end{align}
Likewise, the inequality \eqref{eq:L-bound} implies
\begin{align}
     \|\zz_k -\xx^\star\|_{\MM_k}^2  - \mu_f\|\zz_k-\xx^\star\|_2^2 &\leq 
     \|\zz_k-\xx^\star\|_{\MM_k}^2 - C^{-1}\mu_f\|\zz_k -\xx^\star\|_{\MM_k}^2  \\
     &= (1-C^{-1}\mu_f)\|\zz_k -\xx^\star\|_{\MM_k}^2.
\end{align}
Combining these inequalities with \eqref{theineq} gives
\[
\begin{aligned}
&\|\zz_{k+1} - \xx^\star\|_{\MM_k}^2 + (1 + c\mu_h)\|\uu_{k+1}-\uu^\star \|_{\MM_k^{-1}}^2
\\
&\qquad\leq (1 - C^{-1}\mu_f)\|\zz_k-\xx^\star\|_{\MM_k}^2 + \|\uu_k - \uu^\star\|_{\MM_k^{-1}}^2\\
&\qquad\leq\max\{1-C^{-1}\mu_f,(1+c\mu_h)^{-1}\}\,E_k\\
&\qquad=(1-\min\{\phi,\theta\})\,E_k,
\end{aligned}
\]
where
$\phi = C^{-1}\mu_f$ and $\theta = c\mu_h / (1 + c \mu_h)$.  Converting to $\MM_{k+1}$ using \eqref{eq:big-delta-def} gives
\[
E_{k+1} \leq \Delta_{k,k+1} (1 - \min\{\phi,\theta\})\, E_k.
\]
Iterating gives
\[
E_{k+1} \leq \left(\prod_{j=0}^{k} \Delta_{j,j+1}\right)(1 - \min\{\phi,\theta\})^{k+1}
E_0 = \Delta_{0,k+1} (1 - \min\{\phi,\theta\})^{k+1} E_0.
\]
\end{proof}

\subsection{Proof of \texorpdfstring{\cref{lem:bb-closeness}}{bbcloseness} (bound for Barzilai--Borwein updates)}\label{proof:bb-closeness}

\begin{proof}
For $\omega>0$, the objective in \eqref{eq:bb-obj} is coercive and strongly
convex, so it has a unique minimizer over the nonempty closed convex set
$\mathcal{C}$.  Let $\DD_k:=\widehat\MM_{k+1}-\MM_k$.  If $\DD_k=0$, the bound
is immediate.  Otherwise, writing
$\bs{r}_k:=\yy_k-\MM_k\vs_k$, the minimization problem is equivalent to
\[
\DD_k = \argmin_{\DD \in \mathcal{C} - \MM_k} \tfrac{1}{2} \|\DD \vs_k - \bs{r}_k\|_2^2 + \tfrac{\omega}{2}\|\DD\|_F^2.
\]
Since $\MM_k \in \mathcal{C}$, testing the first-order optimality condition at the feasible point $\DD = 0$ gives
\[
\langle (\DD_k \vs_k - \bs{r}_k)\vs_k^T + \omega \DD_k,\; 0 - \DD_k \rangle_F \geq 0,
\]
i.e., $\omega\|\DD_k\|_F^2 + \|\DD_k \vs_k\|_2^2 \leq \langle \bs{r}_k \vs_k^T, \DD_k\rangle_F \leq \|\bs{r}_k\|_2 \|\vs_k\|_2 \|\DD_k\|_F$ by Cauchy--Schwarz.  Dropping $\|\DD_k \vs_k\|_2^2$ and dividing by $\omega\|\DD_k\|_F$ gives
\[
\|\DD_k\|_F \leq \omega^{-1} \|\yy_k - \MM_k \vs_k\|_2 \|\vs_k\|_2
\leq \omega^{-1}(\|\yy_k\|_2  + \|\MM_k \vs_k\|_2) \|\vs_k\|_2
\leq \omega^{-1}(L + \|\MM_k\|_2) \|\vs_k\|_2^2,
\]
using $\|\yy_k\|_2 \leq L\|\vs_k\|_2$ ($L$-smoothness) in the last step.
Define $\eta_k:=\|\MM_k^{-1/2}\DD_k\MM_k^{-1/2}\|_2$.  The preceding estimate gives
\[
\eta_k\leq\frac{\|\DD_k\|_F}{\lambda_{\min}(\MM_k)}
\leq\frac{L+\|\MM_k\|_2}{\omega\lambda_{\min}(\MM_k)}\|\vs_k\|^2<1.
\]
Thus
\[
(1-\eta_k)\MM_k\preceq\widehat\MM_{k+1}\preceq(1+\eta_k)\MM_k,
\]
and since $1+\eta_k\leq(1-\eta_k)^{-1}$, the definition of $\Delta$ yields
\[
\Delta(\MM_k,\widehat\MM_{k+1})\leq(1-\eta_k)^{-1}
\leq \left(1 - \frac{L + \|\MM_k\|_2}{\omega\lambda_{\min}(\MM_k)}\|\vs_k\|^2\right)^{-1}.
\]
\end{proof}

\subsection{Proof of \texorpdfstring{\cref{cor:truly-linear}}{the fixed-regularization corollary} (removing the polylogarithmic factor)}\label{proof:truly-linear}

\begin{proof}
By \cref{cor:quasilinear} and equivalence of the fixed $\MM_0$ norm with the Euclidean norm,
\[
\|\vs_k\|=\|\zz_{k+1}-\zz_k\|\in O((\log k)^\rho\beta^k),
\qquad\sum_{k=0}^{\infty}\|\vs_k\|^2<\infty.
\]
For each $k$, the feasible set in \eqref{eq:bb-obj-with-bounds} is closed and convex and contains $\MM_k$.  Apply \cref{lem:bb-closeness} to that set, and write $q_k:=\Delta(\MM_k,\widehat\MM_{k+1})$.  Since \eqref{eq:metric-double-bounds} bounds the coefficient in the lemma by $A:=(L+C)/(\omega c)$, its small-step hypothesis holds for all sufficiently large $k$.  For such $k$ with $A\|\vs_k\|^2\leq1/2$,
\[
0\leq q_k-1\leq\frac{A\|\vs_k\|^2}{1-A\|\vs_k\|^2}
\leq2A\|\vs_k\|^2.
\]
Hence $\sum_k(q_k-1)<\infty$ and
\[
Q:=\prod_{k=0}^{\infty}q_k<\infty.
\]

It remains to account for backtracking.  Let $\ell_{k+1}$ be the number of backtracking steps applied to the proposal $\widehat\MM_{k+1}$, so $\MM_{k+1}=\tau^{\ell_{k+1}}\widehat\MM_{k+1}$.  Since $\widehat\MM_{k+1}\succeq q_k^{-1}\MM_k$, induction gives, for every $N\geq1$,
\[
\MM_N\succeq
\frac{\tau^{\sum_{k=0}^{N-1}\ell_{k+1}}}{\prod_{k=0}^{N-1}q_k}\MM_0.
\]
Using $\MM_0\succeq c\II$ and $\MM_N\preceq C\II$ from \eqref{eq:metric-double-bounds},
\[
\tau^{\sum_{k=0}^{N-1}\ell_{k+1}}
\leq\frac{C}{c}\prod_{k=0}^{N-1}q_k
\leq\frac{C}{c}Q.
\]
Because $\tau>1$ and the $\ell_k$ are nonnegative integers, only finitely many backtracking steps occur.  In particular,
\[
\Delta_{0,N}
\leq\prod_{k=0}^{N-1}\tau^{\ell_{k+1}}q_k
\leq\frac{C}{c}Q^2.
\]
Substituting this bound into \cref{thm:linear-convergence} gives $E_N=O(\beta^{2N})$ and hence $d((\zz_N,\uu_N))=O(\beta^N)$.
\end{proof}

\subsection{Proof of \texorpdfstring{\cref{lem:practical-bb-bounds}}{uniform bounds for practical BB clipping}}\label{proof:practical-bb-bounds}

\begin{proof}
For a nonzero secant, strong convexity and cocoercivity of the gradient give
\[
\langle\vs_k,\yy_k\rangle\geq\mu_f\|\vs_k\|^2,
\qquad
\|\yy_k\|^2\leq L\langle\vs_k,\yy_k\rangle.
\]
Together with Cauchy--Schwarz, these inequalities imply
\[
\mu_f\leq a_k
=\frac{\langle\vs_k,\yy_k\rangle}{\|\vs_k\|^2}
\leq b_k
=\frac{\|\yy_k\|^2}{\langle\vs_k,\yy_k\rangle}
\leq L.
\]
The retained intervals for zero secants therefore also satisfy
\[
0<\underline a\leq a_k\leq b_k\leq\overline b<\infty,
\qquad
\underline a:=\min\{\mu_f,a_{\mathrm{init}}\},\quad
\overline b:=\max\{L,b_{\mathrm{init}}\}.
\]
Recall that $\tau>1$ is the backtracking factor and $\sigma\in(0,1)$ is the sufficient decrease parameter in \vmtos.
For an s.p.d.\ matrix, write $\kappa(\MM):=\lambda_{\max}(\MM)/\lambda_{\min}(\MM)$, and define
\[
\begin{aligned}
R&:=\max\{\kappa(\MM_0),\overline b/\underline a\},\\
c&:=\min\{\lambda_{\min}(\MM_0),\underline a\},\\
C&:=\max\{\lambda_{\max}(\MM_0),\overline b,\tau LR/\sigma\}.
\end{aligned}
\]
We show that these constants give the claimed bounds on accepted metrics.

First consider the condition number.  Write $t_k^i:=\operatorname{clip}_{[a_k,b_k]}(p_k^i)$ and
\[
T_d(m,t):=\operatorname{clip}_{[m/d,\,dm]}(t),\qquad d\geq1.
\]
The map $T_d$ is nondecreasing in both arguments and satisfies $T_d(rm,rt)=rT_d(m,t)$ for $r>0$.
Suppose $\kappa(\MM_k)\leq R$.  For every pair of coordinates $i,j$, we have
$m_k^i\leq Rm_k^j$ and $t_k^i\leq\overline b\leq R\underline a\leq Rt_k^j$.  Hence
\[
\widehat m_{k+1}^i
=T_{d_{k+1}}(m_k^i,t_k^i)
\leq T_{d_{k+1}}(Rm_k^j,Rt_k^j)
=R\widehat m_{k+1}^j.
\]
Thus $\kappa(\widehat\MM_{k+1})\leq R$.  Backtracking multiplies all entries by the same positive scalar, preserving the condition number.  Induction gives $\kappa(\MM_k)\leq R$ for every $k$.

Since $[m/d,dm]$ contains $m$, the value $T_d(m,t)$ lies between $m$ and $t$.  Therefore
\[
\min\{m_k^i,\underline a\}\leq\widehat m_{k+1}^i
\leq\max\{m_k^i,\overline b\}.
\]
Backtracking can only increase the diagonal entries, so induction gives $\lambda_{\min}(\MM_k)\geq c$.
If no backtracking occurs at iteration $k+1$, then
\[
\lambda_{\max}(\MM_{k+1})\leq\max\{\lambda_{\max}(\MM_k),\overline b\}.
\]
If backtracking occurs at iteration $k+1$, the last rejected metric is $\MM_{k+1}/\tau$.  By $L$-smoothness, any trial metric with smallest eigenvalue at least $L/\sigma$ satisfies the sufficient decrease condition (see \cref{prop:backtrack}).  Thus
\[
\lambda_{\min}(\MM_{k+1})<\frac{\tau L}{\sigma},
\qquad
\lambda_{\max}(\MM_{k+1})
\leq R\lambda_{\min}(\MM_{k+1})
<\frac{\tau LR}{\sigma}.
\]
Combining the two cases with the initial bound yields $\lambda_{\max}(\MM_k)\leq C$ for every $k$.  Together with the lower bound, this shows that the accepted metrics satisfy \eqref{eq:metric-double-bounds}.
\end{proof}

\section{Diagonal Metric Implementation}\label{supp:diagvm}

\begin{remark}[BB clipping alone does not enforce the closeness bound]\label{rem:bb-clipping-counterexample}
Clipping each $p_k^i$ only to the BB interval $[a_k,b_k]$ need not satisfy the closeness bound in line \ref{line:adapt-metric}.  Take
\[
\mm_k=(1,2),\qquad \vs_k=(1,1),\qquad \yy_k=(2,1),
\qquad \omega_k=\tfrac1{10},\qquad d_{k+1}=\tfrac32.
\]
These secants arise from the convex function $f(\zz)=z_1^2+\tfrac12z_2^2$.
The BB interval is $[3/2,5/3]$, and the unconstrained minimizer is $(p_k^1,p_k^2)=(21/11,12/11)$.
Clipping $\boldsymbol{p}_k$ only to the BB interval gives $(5/3,3/2)$, for which
\[
\Delta\!\left(\MM_k,\mathrm{diag}(5/3,3/2)\right)
=\frac53>\frac32=d_{k+1}.
\]
The closeness intervals are $[2/3,3/2]$ and $[4/3,3]$.  Applying the second clipping step in \eqref{eq:sequential-clipping} returns $(3/2,3/2)$, which satisfies both the BB and closeness bounds.

The two bounds may also be incompatible.  If $m_k^i=10$, $d_{k+1}=1.25$, and $[a_k,b_k]=[2,4]$, the closeness interval is $[8,12.5]$.
These intervals are disjoint, so no value satisfies both bounds.
Our proposed sequential clipping rule returns $8$, satisfying the closeness bound while lying outside the BB interval.
\end{remark}

\paragraph{Weighted proximal operators}
For $\MM=\mathrm{diag}(\mm)$ with $\mm>0$, we compute:
\begin{itemize}
\item \emph{Separable functions} ($\ell_1$, box constraints): compute the weighted prox coordinatewise.
\item \emph{1D total variation}: Johnson's dynamic program solves $\min_{\zz} \tfrac{1}{2}\sum_i m_i (z_i-x_i)^2+\lambda\sum_i|z_{i+1}-z_i|$ in $O(n)$ \citep{johnson2013dynamic,barbero2018modular}.  For a nonnegativity constraint, clip the result at zero.  This gives the constrained prox by \citet[Theorem~1]{yu2013decomposition} in the $\MM$-inner product: the projection $P_+$ preserves adjacent ordering, so $\partial\,\mathrm{TV}(\xx)\subseteq\partial\,\mathrm{TV}(P_+(\xx))$.
\item \emph{Group $\ell_2$ norms}: for $\lambda>0$, the weighted prox of $\lambda\|\cdot\|$ on a group $G$ is zero if and only if $\sum_{i\in G}(m_i x_i/\lambda)^2\leq1$.  Otherwise, $z_i=x_i t/(t+\lambda/m_i)$, where $t=\|\zz_G\|>0$ solves $\sum_{i\in G}x_i^2/(t+\lambda/m_i)^2=1$.  We find this unique root by safeguarded root finding.  For $\lambda=0$, $\zz_G=\xx_G$.  A metric constant within each group gives block soft thresholding \citep[\S 3.2]{park2020variable_arxiv}.
\item \emph{Unit simplex}: compute the weighted projection by sorted breakpoint search \citep{kiwiel2008breakpoint}.
\item \emph{Halfspace}: the weighted projection has a closed form \citep[Example~29.20]{Bauschke2017}.
\end{itemize}

\end{document}